\documentclass[11pt]{article}
\usepackage{amsmath}
\usepackage{amssymb}
\usepackage{amsthm}
\usepackage{mathrsfs}
\usepackage{geometry}
\usepackage{stmaryrd}
\usepackage{tikz}
\usepackage{xcolor}
\usepackage{appendix}
\usepackage{titlesec}
\usepackage{authblk}

\usepackage{url}
\usepackage[colorlinks,
            linkcolor=blue,
            urlcolor=green,
            citecolor=red]{hyperref}

\numberwithin{equation}{section}
\numberwithin{figure}{section}

\newtheorem{theorem}{Theorem}
\newtheorem{lemma}[theorem]{Lemma} 
\newtheorem{corollary}[theorem]{Corollary}
\newtheorem{proposition}[theorem]{Proposition}

\newtheorem*{remark}{Remark}

\begin{document}
\title{\textbf{Oscillatory Integral Operators Along the Paraboloid and Applications}}
\author{Jianghao Zhang}
\newcommand{\Address}{{
\bigskip
\footnotesize
\textsc{Jianghao Zhang: School of Mathematics, University of Edinburgh, Edinburgh, EH9 3FD, UK}\par\nopagebreak
\textit{Email address}: \texttt{J.Zhang-395@sms.ed.ac.uk}
}}

\date{}

\maketitle

\begin{abstract}
We prove decay estimates for oscillatory integral operators along $(y_1, y_2, y_1^2+y_2^2)$ in the finite field and the continuous settings. In the continuous setting, our result allows a bi-parameter maximal modulation, which extends the one-parameter version previously established by Pierce and Yung. Our approach is different from theirs. Then we discuss three applications of our estimates: a power saving bound for certain pattern finding problem in finite fields, a short approach to $L^p$ boundedness of an one-parameter Pierce—Yung operator, and $L^p$ boundedness of a bi-parameter Pierce—Yung operator.
\end{abstract}



\section{Introduction}
\subsection{Background and motivation}
We start with two classical operators in harmonic analysis. The first one is Radon transform of the following type \big($x=(x_{\leq d-1}, x_d)\in \mathbb{R}^{d-1}\times \mathbb{R}$\big):
\begin{equation} \label{EqRadon}
T_{1}f(x):=\int_{\mathbb{R}^{d-1}} f(x_{\leq d-1}-y, x_d- |y|^2) K(y) dy,
\end{equation}
Here and below, $K$ is a Calderon--Zygmund kernel of the form $K(y)=\frac{\Omega(y)}{|y|^{d-1}}$. We assume $\Omega$ is a smooth homogeneous function of degree $0$ and $\int_{\mathbb{S}^{d-2}} \Omega=0$. It is well-known that $\|T_1f\|_p\lesssim \|f\|_p$ for $1<p<\infty$ \cite{Stein}. Second, we have the Stein--Wainger operator:
\begin{equation}  \label{EqSW}
T_2f(x):=\sup_{\tau=(\tau_{\alpha})_{2\leq |\alpha|\leq n}} |\int_{\mathbb{R}^{d}} f(x_{\leq d-1}-y, x_d-y_d) e^{2\pi iP_{\tau}(y)}K(y) dy|,
\end{equation}
where $P_{\tau}$ is the polynomial without linear terms:
\begin{equation*}
P_{\tau}(y):=\sum_{2\leq |\alpha|\leq n} \tau_{\alpha}y^{\alpha}. 
\end{equation*}
Stein and Wainger introduced \eqref{EqSW} in \cite{SWPolynomialCarleson} as a variant of the celebrated Carleson--Hunt theorem \cite{CarlesonCarleson, HuntCarleson}, which asserts that the maximal modulated singular integral
\begin{equation}  \label{EqCarleson}
\sup_{\tau} |\int_{\mathbb{R}} f(x-y) e^{2\pi i\tau y}K(y) dy|
\end{equation}
is $L^p$ bounded for $1<p<\infty$. We emphasize that \eqref{EqCarleson} does not fall under the scope of \eqref{EqSW} because the latter excludes the linear term in $P_{\tau}$. These two operators were handled by rather different methods in the above references. Using a $TT^*$ argument, Stein and Wainger established the $L^p$ boundedness of \eqref{EqSW}: $\|T_2f\|_p\lesssim \|f\|_p, 1<p<\infty$. After \cite{SWPolynomialCarleson}, there are further results which unify the Carleson--Hunt theorem and the Stein--Wainger theorem \cite{LiePolynomialCarleson, ZKPolynomialCarleson}. We refer the reader to these two articles and references therein for a comprehensive introduction.

Following \cite{PYPolynomialCarlesonRadon, AMPYPolynomialCarlesonRadon}, we study an operator which possesses both features of \eqref{EqRadon} and \eqref{EqSW}:
\begin{equation} \label{EqPY}
Tf(x):=\sup_{\tau\in \mathbb{R}} |\int_{\mathbb{R}^2} f(x_{\leq 2}-y, x_3-|y|^2) e^{2\pi i \tau(y_1^3+y_2^3)} K(y) dy|.
\end{equation}
Indeed, \cite{PYPolynomialCarlesonRadon, AMPYPolynomialCarlesonRadon} considered this type of operators in a very general form. Here we focus on the model case $Tf$ defined in \eqref{EqPY}.
\begin{theorem}[Pierce--Yung {\cite[Theorem 1.1]{PYPolynomialCarlesonRadon}}] \label{TheoremPY}
For $1<p<\infty$ and $T$ defined in \eqref{EqPY}, we have
\begin{equation*}
\|Tf\|_p\lesssim \|f\|_p,\ \ \ \forall f\in \mathcal{S}(\mathbb{R}^3).
\end{equation*}
\end{theorem}
\noindent Pierce and Yung proved Theorem \ref{EqPY} by establishing the following decay estimate for oscillatory integral operators along $(y_1, y_2, |y|^2)$:
\begin{lemma}[Pierce--Yung {\cite[Theorem 1.2]{PYPolynomialCarlesonRadon}}] \label{LemmaMaximalMajorArc}
Let $\lambda\geq 1$. Let $P_3\in \mathbb{R}[Y_1, Y_2]$ be a homogeneous polynomial of degree $3$. Then there exists an absolute constant $c>0$ such that
\begin{equation*}  
\big{\|}\sup_{R>0} \sup_{|\tau|\sim \lambda R^{-3}} |\int_{\mathbb{R}^2} f(x_{\leq 2}-y, x_3-|y|^2) e^{2\pi i \tau P_3(y)} \rho_R (y) dy| \big{\|}_2\lesssim \lambda^{-c}\|f\|_2, \ \ \ \forall f\in \mathcal{S}(\mathbb{R}^3).
\end{equation*}
Here $\rho_R(y):=\frac{1}{R^2}\rho(\frac{y}{R})$, and $\rho\in \mathcal{S}(\mathbb{R}^3)$ is supported in $\{|y|<5\}$ with $\|\rho\|_{C^1}\leq 1$.
\end{lemma}
\noindent Below, we will borrow the terminology ``major arc lemma'' from additive combinatorics \cite[Lemma 4.61]{KMPWWMultilinearPolynomialErgodic} and refer to estimates like Lemma \ref{LemmaMaximalMajorArc} as major arc lemmas \footnote{When we apply this type of estimates directly, the oscillation parameter $\lambda>0$ usually lies in the minor arc (i.e. being large). Thus it makes sense to call them minor arc lemmas instead. However, in the additive combinatorics setting, we often apply them in the contrapositive, so that the conclusion becomes ``$\lambda$ lies in the major arc''. This is why one uses the terminology ``major arc lemma'' in additive combinatorics. In our work, we chose to follow the terminology from additive combinatorics literature and use ``major arc lemma'' for all such estimates.}. In particular, we use the name maximal major arc lemma for Lemma \ref{LemmaMaximalMajorArc} because there is a supremum over $R>0$.
\noindent Once we know Lemma \ref{LemmaMaximalMajorArc} for $P_3(y):=y_1^3+y_2^3$, it is straightforward to deduce Theorem \ref{TheoremPY}. See \cite[Section 5]{PYPolynomialCarlesonRadon} or \cite[Section 5]{SWPolynomialCarleson} for this type of deduction. Pierce and Yung's approach \cite{PYPolynomialCarlesonRadon} to Lemma \ref{LemmaMaximalMajorArc} is quite involved. As explained in \cite{PYPolynomialCarlesonRadon}, one cannot directly apply the $TT^*$ argument introduced in \cite{SWPolynomialCarleson} to establish Lemma \ref{LemmaMaximalMajorArc} due to its maximal Radon behavior (i.e. the supremum over $R$ above). Instead, they compare the operator in Lemma \ref{LemmaMaximalMajorArc} with its linear-average analogue corresponding to \eqref{EqSW} , carefully control the difference by a square function, and finally apply the known result for that linear-average analogue established in \cite{SWPolynomialCarleson}. 

\subsection{A new major arc lemma and applications}
We propose a different approach to Theorem \ref{TheoremPY}. We still aim to establish a major arc lemma. However, we do not take supremum over $R$ as Lemma \ref{LemmaMaximalMajorArc}. This enables us to directly use the $TT^*$ method without comparing our parabolic average with the linear average. By carefully exploiting the information from this direct $TT^*$ approach, we are able to establish the following new bi-parameter non-maximal major arc lemma:
\begin{lemma}  \label{LemmaContinuousMajorArcBi-Parameter}
Let $\lambda\gtrsim 1$. Let $\tau_{1, (\cdot)}, \tau_{2, (\cdot)}: \mathbb{R}^3\rightarrow \mathbb{R}$ be arbitrary measurable functions such that $\max\{|\tau_{1, x}|, |\tau_{2, x}| \}>\lambda, \forall x\in \mathbb{R}^3$. Let
\begin{equation*}
Af(x):=\int_{\mathbb{R}^2} f(x_{\leq 2}-y, x_3-|y|^2) e^{2\pi i (\tau_{1, x}y_1^3+\tau_{2, x}y_2^3)} \rho (y) dy,
\end{equation*}
where $\rho\in \mathcal{S}(\mathbb{R}^3)$ is supported in $\{|y|<5\}$ with $\|\rho\|_{C^1}\leq 1$. Then there exists an absolute constant $c>0$ such that
\begin{equation*}
\|Af\|_2\lesssim \lambda^{-c}\|f\|_2, \ \ \ \forall f\in L^2(\mathbb{R}^3).
\end{equation*}
\end{lemma}
\noindent In \cite{SWPolynomialCarleson, PYPolynomialCarlesonRadon, AMPYPolynomialCarlesonRadon}, the authors applied $TT^*$ in the physical space. An interesting aspect of our approach is that we take the Fourier transform before using $TT^*$, which gives a more manageable kernel.

Then we will give a new framework to deduce $L^p$ boundedness of Pierce—Yung type operators from non-maximal major arc lemmas. This provides a short proof of Theorem \ref{TheoremPY}. Moreover, we can establish the following new result:
\begin{theorem}  \label{TheoremPYBi-Parameter}
Define the bi-parameter Pierce--Yung operator:
\begin{equation*}
\mathfrak{T}f(x):=\sup_{\tau_1, \tau_2\in \mathbb{R}} |\int_{\mathbb{R}^2} f(x_{\leq 2}-y, x_3-|y|^2) e^{2\pi i (\tau_1 y_1^3+\tau_2 y_2^3)} K(y) dy|.
\end{equation*}
Then for $1<p<\infty$ we have
\begin{equation*}
\|\mathfrak{T}f\|_p\lesssim \|f\|_p,\ \ \ \forall f\in \mathcal{S}(\mathbb{R}^3).
\end{equation*}
\end{theorem}
\noindent Note that the general results in \cite{PYPolynomialCarlesonRadon, AMPYPolynomialCarlesonRadon} do not allow taking a bi-parameter maximal modulation with polynomials of the same degree.

We remark that Zhu \cite{ZhuPierce--YungHomogeneous} independently established a more general form of Theorem \ref{TheoremPYBi-Parameter}. The argument in \cite{ZhuPierce--YungHomogeneous} is based on the original approach introduced in \cite{PYPolynomialCarlesonRadon, AMPYPolynomialCarlesonRadon}, which differs from ours.

\subsection{The finite field analogue and pattern finding}
Our formulation of (non-maximal) major arc lemmas is inspired by pattern finding problems. In \cite{PelusePolynomialSzemerediFiniteFields, PPNonlinearRoth}, the authors introduced the Gowers norm control and degree-lowering techniques to find polynomial patterns in finite fields or the integers. Their techniques were further developed in many later works including \cite{PelusePolynomialSzemeredi, PPSSquareCorner, KucaConcatenationFiniteFields, KKLConcatenation}. As a milestone, \cite{KMPWWMultilinearPolynomialErgodic} established a very general inverse theorem for higher-dimensional polynomial corner patterns with distinct degrees and use it (together with other deep tools in discrete harmonic analysis) to solve an important case of the Furstenberg--Bergelson--Leibman conjecture \cite[Conjecture 1.20]{BMSWMulti-ParameterPolynomialErgodic}.

A crucial step in the degree-lowering argument is to establish certain major arc lemma similar in spirit to Lemma \ref{LemmaContinuousMajorArcBi-Parameter}. Take the pattern $(x_1, x_2), (x_1-y^2, x_2), (x_1, x_2-y^3)$ in $\mathbb{F}_q^2$ as an example. Then the corresponding major arc lemma roughly says
\begin{equation} \label{EqMajorArcStandard}
\mathbb{E}_{x_1\in \mathbb{F}_q}| \mathbb{E}_{y\in \mathbb{F}_q}g(x_1-y^2)e_{q}(\tau_{x_1}\cdot y^3)|    
\end{equation}
is small unless $\tau_{x_1}$ lies in the major arc (which is $\{0\}$ in finite fields) for a positive portion of $x_1$. Now, if we want to find the pattern $\big{(}x_1, x_2, x_3, x_4\big{)}, \big{(}x_{\leq 3}-(y_1, y_2, y_1^2+y_2^2), x_4 \big{)}, \big{(} x_{\leq 3}, x_4-(y_1^3+y_2^3) \big{)}$ in $\mathbb{F}_q^4$, we are led to consider
\begin{equation} \label{EqMajorArcRadon}
\mathbb{E}_{x_{\leq 3}\in \mathbb{F}_q^3} \big| \mathbb{E}_{y\in \mathbb{F}_q^2}g \big{(}x_{\leq 3}-(y_1, y_2, y_1^2+y_2^2)\big{)}e_{q}\big{(}\tau_{x_{\leq 3}}\cdot (y_1^3+y_2^3)\big{)} \big|.
\end{equation}
One can see the similarity between \eqref{EqMajorArcRadon} and Lemma \ref{LemmaContinuousMajorArcBi-Parameter} after taking $\tau_{1, (\cdot)}=\tau_{2, (\cdot)}$ in the latter. For \eqref{EqMajorArcRadon}, the variables $x_1, x_2, x_3$ in the single function $g$ are averaged by polynomials with possibly different degrees. This feature makes \eqref{EqMajorArcRadon} lie outside the scope of the very general major arc lemma established in \cite{KucaPolynomialCornerFiniteFields, KucaConcatenationFiniteFields}, and it is unclear how to adapt the methods in \cite{PPSSquareCorner, KMPWWMultilinearPolynomialErgodic, KucaPolynomialCornerFiniteFields, KucaConcatenationFiniteFields} to this setting. Nevertheless, we give a unified $TT^*$ proof of this type of major arc lemmas in the finite field and the continuous settings. Combining this major arc lemma with the Gowers norm control and degree-lowing techniques, we then establish
\begin{theorem}  \label{TheoremFindingPatterns}
Let $q$ be a prime. Let $A$ be a subset of $\mathbb{F}_q^4$ with density $\alpha:=\frac{\#A}{q^4}$. Then there is an absolute constant $c_0>0$ such that
\begin{equation}  \label{EqFindingPatterns}
\begin{split}
\frac{1}{q^6}\# \big{\{}(x, y)\in \mathbb{F}_q^4\times \mathbb{F}_q^2: & \big{(}x_{\leq 3}, x_4\big{)}, \big{(}x_{\leq 3}-(y_1, y_2, y_1^2+y_2^2), x_4 \big{)}, \big{(} x_{\leq 3}, x_4-(y_1^3+y_2^3) \big{)}\in A \big{\}}   \\
&=\mathbb{E}_{x\in \mathbb{F}_q^4} 1_A(x)\big{(} \mathbb{E}_{x_{\leq 3}\in \mathbb{F}_q^3} 1_A(\cdot, x_4) \big{)} \big{(} \mathbb{E}_{x_4\in \mathbb{F}_q} 1_A(x_{\leq 3}, \cdot)  \big{)}+O(q^{-c_0}).
\end{split}
\end{equation}
Moreover, the right hand side is at least $\alpha^3+O(q^{-c_0})$.
\end{theorem}

\paragraph{Outline of the paper.}
\begin{itemize}
\item In Section \ref{SectionMajorArcLemma}, we prove our major arc lemmas in the finite field and the continuous settings. We start with the finite field case and eventually establish the continuous case Lemma \ref{LemmaContinuousMajorArcBi-Parameter}.
\item In Section \ref{SectionFindingPatterns}, we prove Theorem \ref{TheoremFindingPatterns} by combining the Gowers norm control result in \cite{KucaPolynomialCornerFiniteFields}, the degree-lowing technique in \cite{PPSSquareCorner}, and our major arc lemma in finite fields.
\item In Section \ref{SectionPY}, we give a short proof of Theorem \ref{TheoremPY} using the one-parameter version of Lemma \ref{LemmaContinuousMajorArcBi-Parameter} (which is non-maximal unlike Lemma \ref{LemmaMaximalMajorArc}).
\item In Section \ref{SectionPYBi-Parameter}, we prove Theorem \ref{TheoremPYBi-Parameter} by further developing the argument in Section \ref{SectionPY}.
\item In Section \ref{SectionFinalRemarks}, we give several remarks including the limitation of our argument.
\end{itemize}

\paragraph{Notations and conventions.}
\begin{itemize}
\item We use $x_{\leq k}$ for the first $k$ coordinates in $x=(x_1, x_2, \cdots, x_d)$.
\item Define
\begin{equation}  \label{EqDefinitionOfGammaAndP}
\Gamma(y):=(y_1, y_2, y_1^2+y_2^2),\ \ \ P(y):=y_1^3+y_2^3.
\end{equation}
\item We use $0<c\ll c_1\ll 1, 0<c_2, c_3<1$ for small absolute constants and $1\ll C_0\ll C_1<\infty$ for large absolute constants. Their precise values are not important in this paper.
\item For two positive quantity $a, b$, the notation $a\sim b$ means $\frac{1}{2}b<a\leq b$.
\item In $\mathbb{F}_q^d$, let $\mathbb{E}_{x\in X}$ denote the average $\frac{1}{q^d}\sum_{x\in X}$. Note that we use the normalized factor $\frac{1}{q^d}$ rather than $\frac{1}{|X|}$ for simplicity. ($|X|$ is comparable with $q^d$ most of the time.) Moreover, we use probability measure on the physical side and counting measure on the frequency side. Hence $\|f(x)\|_{L_x^p}:=(\mathbb{E}_{x\in \mathbb{F}_q^d} |f(x)|^p )^{\frac{1}{p}}$ and $\|\hat{f}(\xi)\|_{L_{\xi}^p}:=(\sum_{\xi\in \mathbb{F}_q^d} |\hat{f}(\xi)|^p )^{\frac{1}{p}}$.
\item In $\mathbb{F}_q^d$, let $e_q(\cdot)$ denote $e^{2\pi i\frac{\cdot}{q}}$. We may omit the prime $q$ when it is clear from the context.
\end{itemize}

\paragraph{AI usage.} We did not use AI in any aspect of this work before uploading its first version to arxiv. After the upload, we used ChatGPT-6 for proofreading. It identified an inaccurate use of integration by parts for \eqref{EqPYBi-ParameterOperatorPartOne4} and several minor typos. The former missed an (one-line but necessary) angular decomposition of the amplitude function $K_0$, which has been added in the current version.

\paragraph{Acknowledgement.} The author would like to thank Jonathan Hickman for some valuable discussions. The author is also grateful to Lillian Pierce for helpful comments about this manuscript.

\section{Major Arc Lemmas Along the Paraboloid} \label{SectionMajorArcLemma}
We shall establish several versions of Lemma \ref{LemmaContinuousMajorArcBi-Parameter} by combining the $TT^*$ method, Schur's test, and exponential sum/oscillatory integral estimates.

\subsection{A one-parameter major arc lemma in the finite field setting}
In the finite field setting, we use the Fourier transform and its inverse defined as follows:
\begin{equation*}
\hat{f}(\xi):=\mathbb{E}_{x\in \mathbb{F}_q^n} e(-\xi\cdot x) f(x),\ \ \ f(x)=\sum_{\xi\in \mathbb{F}_q^n} e(x\cdot \xi) \hat{f}(\xi),
\end{equation*}
and Plancherel reads as $\|f\|_{L_x^2}=\|\hat{f}\|_{L_{\xi}^2}$.

We will need the following exponential sum estimate. It is sufficient to have a positive gain beyond the trivial bound, so one may instead use Weyl differencing to get such an estimate.
\begin{lemma}[Mordell {\cite[Chapter 4]{MontgomeryTenLectures}}] \label{LemmaExponentialSumFiniteFields}
Let $q$ be a prime. Let $Q\in \mathbb{Z}[X]$ be a polynomial with degree $0<d<q$. Suppose $q$ does not divide the leading coefficient of $Q$. Then
\begin{equation*}
\big{|}\mathbb{E}_{t\in [q]} e\big{(} Q(t) \big{)} \big{|}\lesssim_d q^{-\frac{1}{d}}.
\end{equation*}
\end{lemma}
Below is the simplest version of our major arc lemmas.
\begin{lemma}  \label{LemmaFiniteFieldsMajorArc1}
Let $q$ be a prime. Let $\tau_{(\cdot)}: \mathbb{F}_q^3\rightarrow \mathbb{F}_q$ be any function. Suppose that $\tau_x\neq 0, \forall x\in \mathbb{F}_q^3$. Then for
\begin{equation}  \label{EqFiniteFieldsAveragingOperator}
Af(x):=\mathbb{E}_{y\in \mathbb{F}_q^2} f \big( x-\Gamma(y) \big) e\big( \tau_x\cdot P(y) \big)
\end{equation}
we have
\begin{equation*}
\|Af\|_{2}\lesssim q^{-c} \|f\|_2.
\end{equation*}
\end{lemma}
\begin{remark}
Recall the notations defined in \eqref{EqDefinitionOfGammaAndP}.
\end{remark}
\begin{proof}
WLOG assume $q>3$. We write $Af$ into a kernel representation on the Fourier side:
\begin{equation*}
Af(x)=\mathbb{E}_{y\in \mathbb{F}_q^2} \sum_{\xi\in \mathbb{F}_q^3} e \big(\xi\cdot (x-\Gamma(y)) \big)\hat{f}(\xi)e\big( \tau_x\cdot P(y) \big)=\sum_{\xi\in \mathbb{F}_q^3} K(x, \xi)\hat{f}(\xi),
\end{equation*}
where $K(x, \xi):=e(\xi\cdot x)\mathbb{E}_{y\in \mathbb{F}_q^2} e \big(-\xi\cdot\Gamma(y)+\tau_{x}P(y) \big)$. By Plancherel, it suffices to show the following operator has norm $O(q^{-c})$.
\begin{equation*}
f\mapsto \sum_{\xi\in \mathbb{F}_q^3} K(x, \xi)f(\xi)=:\mathfrak{A}f.
\end{equation*}
Then we compute the kernel of $\mathfrak{A}\mathfrak{A}^*$ as
\begin{equation}  \label{EqFiniteFieldsSummand1}
\begin{split}
\widetilde{K}(x, z):&=\sum_{\xi\in \mathbb{F}_q^3} K(x, \xi)\overline{K(z, \xi)} \\
& =\sum_{\xi\in \mathbb{F}_q^3} e\big( (x-z)\cdot\xi \big) \mathbb{E}_{y, y' \in \mathbb{F}_q^2} e \big(-\xi\cdot\Gamma(y)+\tau_{x}P(y) \big)e \big(-\xi\cdot\Gamma(y')+\tau_{x}P(y') \big).
\end{split}
\end{equation}
By Schur's test, it suffices to show
\begin{align}
& \mathbb{E}_{z\in \mathbb{F}_q^3} |\widetilde{K}(x, z)|\lesssim q^{-2c},\ \forall x\in \mathbb{F}_q^3, \label{EqFiniteFieldsSchurTest1}   \\
& \mathbb{E}_{x\in \mathbb{F}_q^3} |\widetilde{K}(x, z)|\lesssim q^{-2c},\ \forall z\in \mathbb{F}_q^3.  \label{EqFiniteFieldsSchurTest2}
\end{align}
Since $\widetilde{K}(x, z)=\overline{\widetilde{K}(z, x)}$, we only need to show \eqref{EqFiniteFieldsSchurTest1}. Fix $x$ from now on. Change the variable $z\mapsto x-z$. We exploit the oscillation by writing \eqref{EqFiniteFieldsSummand1} into:
\begin{equation}  \label{EqFiniteFieldsSummand2}
\begin{split}
\widetilde{K}(x, x-z) &=\mathbb{E}_{y, y'\in \mathbb{F}_q^2}\sum_{\xi\in \mathbb{F}_q^3}  e(z\cdot\xi) e\Big(-\xi\cdot \big( \Gamma(y)-\Gamma(y') \big) \Big) e \big(\tau_{x}P(y)-\tau_{x-z}P(y') \big) \\
&=\mathbb{E}_{y, y'\in \mathbb{F}_q^2} q^31_{\Gamma(y)-\Gamma(y')=z} e \big(\tau_{x}P(y)-\tau_{x-z}P(y') \big).
\end{split}
\end{equation}
Note that $\Gamma(y)-\Gamma(y')=z$ is equivalent to
\begin{equation}  \label{EqFiniteFieldsRangeOfy}
y-y'=z_{\leq 2},\ \ \ y\cdot z_{\leq 2}=2^{-1}(z_3+z_1^2+z_2^2).
\end{equation}
Here $2^{-1}$ is the inverse of $2$ modulo $q$. When $z_{\leq 2}=0$, we have
\begin{equation}  \label{EqFiniteFieldsRDegenerateCase}
\text{RHS of \eqref{EqFiniteFieldsSummand2}}=\begin{cases}
0,\ z_3\neq 0, \\
\frac{1}{q^4} \sum_{y=y'} q^3=q,\ z_3=0
\end{cases}
\end{equation}
When $z_{\leq 2}\neq 0$, \eqref{EqFiniteFieldsRangeOfy} is equivalent to $y-y'=z_{\leq 2}$ and $y$ belongs to the line $\{t\in \mathbb{F}_q: tz_{\leq 2}^{\perp}+b_z\}$. Here $z_{\leq 2}^{\perp}\neq 0, b_z$ are arbitrarily fixed elements in $\mathbb{F}_q^2$ such that $z_{\leq 2}^{\perp}\cdot z_{\leq 2}=0, b_z\cdot z_{\leq 2}=2^{-1}(z_3+z_1^2+z_2^2)$ respectively. Thus
\begin{equation}  \label{EqFiniteFieldsSummand3}
\text{RHS of \eqref{EqFiniteFieldsSummand2}}=\mathbb{E}_{t\in \mathbb{F}_q} e \big(\tau_{x}P(tz_{\leq 2}^{\perp}+b_z)-\tau_{x-z}P(tz_{\leq 2}^{\perp}+b_z-z_{\leq 2}) \big)=:\mathbb{E}_{t\in \mathbb{F}_q} e \big( Q_z(t) \big).
\end{equation}
In particular, \eqref{EqFiniteFieldsSummand3} implies $|\widetilde{K}(x, x-z)|\leq 1$ when $z_{\leq 2}\neq 0$. To prove \eqref{EqFiniteFieldsSchurTest1}, we shall select an exceptional set $\Upsilon\subset \mathbb{F}_q^3\backslash \{z: z_{\leq 2}=0\}$ such that $\#\Upsilon\lesssim q^{-2c}\cdot q^3$ and $|\widetilde{K}(x, x-z)|\lesssim q^{-2c}$ for $z\in \Upsilon_{\neq 0}^c:=\mathbb{F}_q^3\backslash (\Upsilon\cup \{z: z_{\leq 2}=0\})$. With this $\Upsilon$, we can use its properties and \eqref{EqFiniteFieldsRDegenerateCase} to estimate
\begin{equation*}
\begin{split}
\mathbb{E}_{z\in \mathbb{F}_q^3} |\widetilde{K}(x, x-z)| & =\frac{1}{q^3} \sum_{z: z_{\leq 2}=0} |\widetilde{K}(x, x-z)|+\frac{1}{q^3}\sum_{z\in \Upsilon}|\widetilde{K}(x, x-z)|+\frac{1}{q^3}\sum_{z\in \Upsilon_{\neq 0}^c} |\widetilde{K}(x, x-z)| \\
& \lesssim \frac{1}{q^3}\cdot q+\frac{1}{q^3}\#\Upsilon+\frac{1}{q^3}\cdot q^3\cdot q^{-2c}\lesssim q^{-2c},
\end{split}
\end{equation*}
which gives \eqref{EqFiniteFieldsSchurTest1}.

We select $\Upsilon$ to ensure that the polynomial $Q_z(t)$ in RHS of \eqref{EqFiniteFieldsSummand3} has degree strictly larger than zero for $z\in \Upsilon_{\neq 0}^c$. The coefficients of $t^3$ and $t^2$ are
\begin{align}
& (\tau_{x}-\tau_{x-z})\big( (z_{\leq 2}^{\perp})_1^3+(z_{\leq 2}^{\perp})_1^3 \big), \label{EqFiniteFieldsCoefficient1} \\
(\tau_{x}-\tau_{x-z})\big( (z_{\leq 2}^{\perp})_1^2(b_z)_1 & +(z_{\leq 2}^{\perp})_2^2(b_z)_2 \big)+\tau_{x-z}\big( (z_{\leq 2}^{\perp})_1^2z_1+(z_{\leq 2}^{\perp})_2^2z_2 \big), \label{EqFiniteFieldsCoefficient2}
\end{align}
respectively. Let $\Upsilon:=\{z: z_{\leq 2}\neq 0, (z_{\leq 2}^{\perp})_1^3+(z_{\leq 2}^{\perp})_1^3=0 \text{ or } (z_{\leq 2}^{\perp})_1^2z_1+(z_{\leq 2}^{\perp})_2^2z_2=0 \}$. Recall $\tau_x\neq 0, \forall x\in \mathbb{F}_q^3$. We see either $t^3$ or $t^2$ has a non-zero coefficient in $Q_z(t)$ when $z\in \Upsilon_{\neq 0}^c$. In this case, the desired estimate $|\widetilde{K}(x, x-z)|\lesssim q^{-2c}$ follows from Lemma \ref{LemmaExponentialSumFiniteFields}. It remains to check $\#\Upsilon\lesssim q^{-2c}\cdot q^3$. Let $z\in \Upsilon$. WLOG assume $z_1\neq 0$. Then we can write $z_2=rz_1$, and $z_{\leq 2}^{\perp}\neq 0$ is parallel with $(-r, 1)$. Plug these into the definition of $\Upsilon$. We see that $r$ is in the zero sets of $(-r)^3+1$ or $(-r)^2+r$, which means $\Upsilon$ is contained in the union of $O(1)$ many two-dimensional subspaces of $\mathbb{F}_q^3$. This implies $\#\Upsilon\lesssim q^{-1}\cdot q^3$ and completes the proof.
\end{proof}

In the proof of Theorem \ref{TheoremFindingPatterns}, we actually need an inverse version of Lemma \ref{LemmaFiniteFieldsMajorArc1}. It easily follows by combining the above argument and the following standard popularity principle.
\begin{lemma}  \label{LemmaPopularityPrinciple}
Let $0\leq Y\leq 1$ be a random variable in an arbitrary probability space. Let $\lambda\in (0,1)$. Suppose $\mathbb{E}Y>\lambda$. Then $P(Y>\frac{\lambda}{2})>\frac{\lambda}{2}$.
\end{lemma}
\begin{proof}
Assume $P(Y>\frac{\lambda}{2})\leq \frac{\lambda}{2}$. Then
\begin{equation*}
\mathbb{E}Y=\mathbb{E}(Y: Y>\frac{\lambda}{2})+\mathbb{E}(Y: Y>\frac{\lambda}{2})\leq 1\cdot \frac{\lambda}{2}+ \frac{\lambda}{2}\cdot 1=\lambda,
\end{equation*}
which contradicts the condition.
\end{proof}

\begin{lemma}  \label{LemmaFiniteFieldsMajorArc2}
Let $q$ be a prime and $X$ be a subset of $\mathbb{F}_q^3$. Let $\tau_{(\cdot)}: X\rightarrow \mathbb{F}_q$ be any function. Let $\delta\in (0, 1)$ such that $q>C_0\delta^{-C_0}$. Suppose the following holds for some $f$ ($Af$ defined by \eqref{EqFiniteFieldsAveragingOperator}):
\begin{equation*}
(\mathbb{E}_{x\in X} |Af|^2)^{\frac{1}{2}}>\delta\|f\|_2.
\end{equation*}
Then there is $X_0\subset X$ with $\#X_0\gtrsim \delta^{O(1)}q^3$ such that $\tau_{x}=0, \forall x\in X_0$.
\end{lemma}
\begin{proof}
We use the same notations as in the proof of Lemma \ref{LemmaFiniteFieldsMajorArc1}. Note that $(\mathbb{E}_{x\in X} |Af|^2)^{\frac{1}{2}}=\|1_{x\in X}Af\|_2$. Use Plancherel as before. We see the operator $1_{x\in X}A$ with kernel $1_{x\in X}K(x, \xi)$ has operator norm larger than $\delta$, which means $(1_{x\in X}A)(1_{x\in X}A)^*$ has operator norm larger than $\delta^2$. Apply Schur's test in its contrapositive. We get there exists an $x\in \mathbb{F}_q^3$ such that
\begin{equation*}
\mathbb{E}_{z\in \mathbb{F}_q^3} |1_{x\in X}\widetilde{K}(x, z)1_{z\in X}|>\delta^2.
\end{equation*}
Let $z\mapsto x-z$ and split the sum as before. We have
\begin{equation*}
\mathbb{E}_{z\in \mathbb{F}_q^3} |1_{x\in X}\widetilde{K}(x, x-z)1_{x-z\in X}|\leq O(q^{-2c})+\frac{1}{q^3}\sum_{z\in \Upsilon_{\neq 0}^c} |\widetilde{K}(x, x-z)1_{x-z\in X}|.
\end{equation*}
Thus
\begin{equation*}
\frac{1}{q^3}\sum_{z\in \Upsilon_{\neq 0}^c} |\widetilde{K}(x, x-z)1_{x-z\in X}|>\frac{\delta^2}{2}.
\end{equation*}
Then Lemma \ref{LemmaPopularityPrinciple} gives us a set $Z$ with $\#Z>\frac{\delta^2}{4}q^3$ such that for any $z\in Z$ we have $z\in \Upsilon_{\neq 0}^c$, $x-z\in X$, and 
\begin{equation*}
|\widetilde{K}(x, x-z)|>\frac{\delta^2}{4}.
\end{equation*}
For each such $z$, the two coefficients \eqref{EqFiniteFieldsCoefficient1}\eqref{EqFiniteFieldsCoefficient2} must be $0$ simultaneously; otherwise Lemma \ref{LemmaExponentialSumFiniteFields} gives a contradiction. Our selection of $\Upsilon$ then implies $\tau_{x}-\tau_{x-z}=0, \tau_{x-z}=0$. Hence $X_0:=x-Z$ satisfies the requirements.
\end{proof}

\subsection{A one-parameter major arc lemma in the continuous setting}
Then we establish the continuous analogue of Lemma \ref{LemmaFiniteFieldsMajorArc1}, which is weaker than Lemma \ref{LemmaContinuousMajorArcBi-Parameter}. We chose to prove it separately because its proof already contains the main ideas for the continuous setting, and we only need such a weaker version in the short proof of Theorem \ref{TheoremPY}.
\begin{lemma}  \label{LemmaContinuousMajorArc}
Let $\lambda\gtrsim 1$. Let $\tau_{(\cdot)}: \mathbb{R}^3\rightarrow \mathbb{R}$ be any measurable function such that $|\tau_{x}|>\lambda, \forall x\in \mathbb{R}^3$. Let
\begin{equation*}
Af(x):=\int_{\mathbb{R}^2} f \big( x-\Gamma(y) \big) e^{2\pi i \tau_x P(y)} \rho (y) dy,
\end{equation*}
where $\rho\in \mathcal{S}(\mathbb{R}^3)$ is supported in $\{|y|<5\}$ with $\|\rho\|_{C^1}\leq 1$. Then we have
\begin{equation*}
\|Af\|_2\lesssim \lambda^{-c}\|f\|_2, \ \ \ \forall f\in L^2(\mathbb{R}^3).
\end{equation*}
\end{lemma}
\begin{proof}
Note that the the operator $Af$ is local, so it suffices to show $\|Af\|_{L^2(B)}\lesssim \lambda^{-c}\|f\|_2$ for any fixed ball $B$ of radius $1$. By approximation, we can also assume $\hat{f}$ is supported in $\{|\xi|<R\}$ for some large $N$. Let $\{|\xi|\leq 1\}\prec \chi\prec \{|\xi|\leq 2\}$ be a bump function. We only need to show $\big{\|} 1_{B}(x)A \big( \chi(\frac{D}{R})f \big)(x) \big{\|}_2\lesssim \lambda^{-c}\|f\|_2$. Write $1_{B}(x)A \big( \chi(\frac{D}{R})f \big)(x)$ into a kernel representation on the Fourier side and apply Plancherel as before. Our goal becomes to show the following operator has operator norm $O(\lambda^{-c})$:
\begin{equation*}
\mathfrak{A}f:=\int_{\mathbb{R}^3} K(x, \xi) f(\xi) d\xi:=\int_{\mathbb{R}^3} e^{2\pi i x\cdot \xi} 1_B(x) \big( \int_{\mathbb{R}^2} e^{2\pi i\big( -\xi\cdot\Gamma(y)+\tau_{x}P(y)\big) }\rho (y) dy \big) \chi(\frac{\xi}{R}) f(\xi) d\xi.
\end{equation*}
Then we compute the kernel of $\mathfrak{A}\mathfrak{A}^*$ as
\begin{equation*}
\begin{split}
&\widetilde{K}(x, z):=\int_{\mathbb{R}^3} K(x, \xi)\overline{K(z, \xi)} d\xi \\
&=1_{B^2}(x, z)\int_{\mathbb{R}^3} e^{2\pi i( (x-z)\cdot\xi )} \chi^2(\frac{\xi}{R}) \big( \int_{(\mathbb{R}^2)^2} e^{2\pi i\big( -\xi\cdot\Gamma(y)+\tau_{x}P(y)\big) } e^{-2\pi i\big( -\xi\cdot\Gamma(y')+\tau_{z}P(y')\big) }\rho(y)\rho(y') \big) \\
&=1_{B^2}(x, z)\int_{(\mathbb{R}^2)^2} e^{2\pi i\big( \tau_{x}P(y)-\tau_{z}P(y') \big)} (\chi^2)_R^{\vee} \Big( x-z-\big( \Gamma(y)-\Gamma(y') \big) \Big) \rho(y)\rho(y').
\end{split}
\end{equation*}
Apply Schur's test, use the symmetry of $\widetilde{K}(x, z)$, and let $z\mapsto x-z$ as in the proof of Lemma \ref{LemmaFiniteFieldsMajorArc1}. It suffices to show the following for fixed $x$:
\begin{equation}  \label{EqContinuousSchurTest}
\int_{|z|<2} \Big| \int_{(\mathbb{R}^2)^2} e^{2\pi i\big( \tau_{x}P(y)-\tau_{x-z}P(y') \big)} (\chi^2)_R^{\vee} \Big( z-\big( \Gamma(y)-\Gamma(y') \big) \Big) \rho(y)\rho(y') \Big|\lesssim \lambda^{-2c}.
\end{equation}
Let $\widetilde{\mathcal{K}}(x, x-z)$ denote the inner integral $\int_{(\mathbb{R}^2)^2}\cdots$ in LHS of \eqref{EqContinuousSchurTest}. We now change the variable to fix the entry of $(\chi^2)_R^{\vee}(\cdot)$ so that we can exploit the oscillation. Let $w:=y-y'$ and $v:=\frac{w}{|w|}$. (Only need to consider the case $w\neq 0$.) Define $v^{\perp}$ to be the unique unit vector such that $v^{\perp}\cdot v=0$ and $\det (v, v^{\perp})>0$. Let $y=tv^{\perp}+sv=:U(t,s)$. Then
\begin{equation*}
\Gamma(y)-\Gamma(y')=(w, -|w|^2+2y\cdot w)=(w, -|w|^2+2s|w|).
\end{equation*}
Thus
\begin{equation} \label{EqContinuousSummand}
\begin{split}
\widetilde{\mathcal{K}}(x, x-z) & =\int_{w\in \mathbb{R}^2} \int_{s\in \mathbb{R}} (\chi^2)_R^{\vee}\big (z-(w, -|w|^2+2s|w|) \big) I(w, s) \\
&=\int_{0<|w|<10} \int_{|s|\lesssim 1} (\chi^2)_R^{\vee}\big (z-(w, -|w|^2+2s|w|) \big) I(w, s).
\end{split}
\end{equation}
Here
\begin{equation} \label{EqContinuousSummandOscillatoryPart}
I(w, s):=\int_{t\in \mathbb{R}} e^{2\pi i\big( \tau_{x}P(tv^{\perp}+sv)-\tau_{x-z}P(tv^{\perp}+sv-w) \big) } \rho\big( U(t, s) \big)\rho\big( U(t, s)-w \big).
\end{equation}
We shall select an exceptional set $\Upsilon\subset \{0<|w|<10\}$ such that $|\Upsilon|\lesssim \lambda^{-2c}$ and $|I(w, s)|\lesssim \lambda^{-2c}$ for $w\in \{0<|w|<10\}\backslash \Upsilon, |s|\lesssim 1$. With this $\Upsilon$, we can then use its properties to estimate
\begin{equation*}
|\widetilde{\mathcal{K}}(x, x-z)|\lesssim \int_{0<|w|<10} \int_{|s|\lesssim 1} \big| (\chi^2)_R^{\vee}\big (z-(w, -|w|^2+2s|w|) \big) \big| \big(\lambda^{-2c}+1_{\Upsilon}(w) \big),
\end{equation*}
so that $\text{LHS of \eqref{EqContinuousSchurTest}}=\int_{|z|<2} |\widetilde{\mathcal{K}}(x, x-z)|$ is bounded by
\begin{equation*}
\begin{split}
\int_{|z|<2} \int_{0<|w|<10} & \int_{|s|\lesssim 1}\big| (\chi^2)_R^{\vee}\big (z-(w, -|w|^2+2s|w|) \big) \big|\big(\lambda^{-2c}+1_{\Upsilon}(w) \big) \\
&\lesssim \lambda^{-2c}+\int_{|s|\lesssim 1} \int_{0<|w|<10} \int_{|z|<2}\big| (\chi^2)_R^{\vee}\big (z-(w, -|w|^2+2s|w|) \big) \big| 1_{\Upsilon}(w)  \\
&\lesssim \lambda^{-2c}+\int_{|s|\lesssim 1} \int_{0<|w|<10} 1_{\Upsilon}(w)\lesssim \lambda^{-2c}
\end{split}
\end{equation*}
as desired. 

We select $\Upsilon$ to ensure that the polynomial $Q_{w, s}(t):=\tau_{x}P(tv^{\perp}+sv)-\tau_{x-z}P(tv^{\perp}+sv-w)$ in RHS of \eqref{EqContinuousSummandOscillatoryPart} has a large cubic or quadratic coefficient for $w\in \{0<|w|<10\}\backslash\Upsilon$. The coefficients of $t^3$ and $t^2$ are
\begin{align}
& (\tau_{x}-\tau_{x-z})\big( (v_1^{\perp})^3+(v_2^{\perp})^3 \big), \label{EqContinuousCoefficient1} \\
(\tau_{x}-\tau_{x-z})\big( (v_1^{\perp})^2sv_1 & +(v_2^{\perp})^2sv_2 \big)+\tau_{x-z}\big( (v_1^{\perp})^2w_1+(v_2^{\perp})^2w_2 \big), \label{EqContinuousCoefficient2}
\end{align}
respectively. Recall $|\tau_{x}|>\lambda, \forall x\in \mathbb{R}^3$. Let $\Upsilon:=\{0<|w|<10: |(v_1^{\perp})^3+(v_2^{\perp})^3|\leq \lambda^{-c_1} \text{ or } |(v_1^{\perp})^2w_1+(v_2^{\perp})^2w_2|\leq \lambda^{-c_1} \}$. Let $w\in \{0<|w|<10\}\backslash\Upsilon$. Then if $|\tau_{x}-\tau_{x-z}|>\lambda^{2c_1}$, we have $|\eqref{EqContinuousCoefficient1}|>\lambda^{c_1}$; and if $|\tau_{x}-\tau_{x-z}|\leq \lambda^{2c_1}$, we have $|\eqref{EqContinuousCoefficient2}|>\lambda^{1-c_1}-O(\lambda^{2c_1})\gtrsim \lambda^{1-c_1}$. In each case, the estimate $|I(w, s)|\lesssim \lambda^{-2c}$ follows from the van der Corput lemma \cite[Proposition 2.1]{SWPolynomialCarleson}. It remains to check $|\Upsilon|\lesssim \lambda^{-2c}$. This is straightforward via the polar coordinate system $w=r(\cos \theta, \sin \theta)$. We have $v^{\perp}=(-\sin \theta, \cos \theta)$. Thus
\begin{align*}
|(v_1^{\perp})^3+(v_2^{\perp})^3|\leq \lambda^{-c_1} & \Longrightarrow |-\sin \theta+\cos \theta|\lesssim \lambda^{-c_1}, \\
|(v_1^{\perp})^2w_1+(v_2^{\perp})^2w_2|\leq \lambda^{-c_1}& \Longrightarrow |\sin \theta|\lesssim \lambda^{-\frac{c_1}{3}} \text{ or } |\cos \theta|\lesssim \lambda^{-\frac{c_1}{3}} \text{ or } |\sin \theta+\cos \theta|\lesssim \lambda^{-\frac{c_1}{3}},
\end{align*}
which implies $|\Upsilon|\lesssim \lambda^{-\frac{c_1}{3}}$ and completes the proof.
\end{proof}

\subsection{A bi-parameter major arc lemma in the continuous setting}
Now we prove Lemma \ref{LemmaContinuousMajorArcBi-Parameter}. By decompose $Af$ according to the magnitude of $\max\{|\tau_{1, x}|, |\tau_{2, x}| \}$, we can assume $\lambda<\max\{|\tau_{1, x}|, |\tau_{2, x}| \}\leq 2\lambda, \forall x\in \mathbb{R}^3$. Do the same localization procedure, apply Schur's test, use the symmetry of $\widetilde{K}(x, z)$, and let $z\mapsto x-z$ as in the proof of Lemma \ref{LemmaContinuousMajorArc}. It suffices to show the following analogue of \eqref{EqContinuousSchurTest} for fixed $x$:
\begin{equation}  \label{EqContinuousBi-ParameterSchurTest}
\int_{|z|<2} \Big| \int_{(\mathbb{R}^2)^2} e^{2\pi i\big( ( \tau_{1, x}y_1^3+\tau_{2, x}y_2^3)-(\tau_{1, x-z}{y'_1}^3+\tau_{2, x-z}{y'_2}^3) \big) } (\chi^2)_R^{\vee} \Big( z-\big( \Gamma(y)-\Gamma(y') \big) \Big) \rho(y)\rho(y') \Big|\lesssim \lambda^{-2c}.
\end{equation}
As there, let $\widetilde{\mathcal{K}}(x, x-z)$ denote the inner integral $\int_{(\mathbb{R}^2)^2}\cdots$ in LHS of \eqref{EqContinuousBi-ParameterSchurTest}. We have the same expression \eqref{EqContinuousSummand} for $\widetilde{\mathcal{K}}(x, x-z)$ with
\begin{equation} \label{EqContinuousBi-ParameterSummandOscillatoryPart}
I(w, s):=\int_{t\in \mathbb{R}} e^{2\pi iQ_{w, s}(t) } \rho\big( U(t, s) \big)\rho\big( U(t, s)-w \big).
\end{equation}    
The only difference is the definition of $Q_{w, s}(t)$:
\begin{equation*}
Q_{w, s}(t):=\tau_{1, x}(tv_1^{\perp}+sv_1)^3+\tau_{2, x}(tv_2^{\perp}+sv_2)^3-\tau_{1, x-z}(tv_1^{\perp}+sv_1-w_1)^3-\tau_{2, x-z}(tv_2^{\perp}+sv_2-w_2)^3.
\end{equation*}
We shall select an exceptional set $\Upsilon\subset \{(w, s): 0<|w|<10, |s|\lesssim 1\}$ such that $|\Upsilon|\lesssim \lambda^{-2c}$ and $|I(w, s)|\lesssim \lambda^{-2c}$ for $(w, s)\in \{(w, s): 0<|w|<10, |s|\lesssim 1\}\backslash\Upsilon$. With this $\Upsilon$, we can argue as in the proof of Lemma \ref{LemmaContinuousMajorArc} to get \eqref{EqContinuousBi-ParameterSchurTest}. For $k=1, 2, 3$, let $a_{w, s, k}$ denote the coefficient of $t^k$ in the polynomial $Q_{w, s}(t)$. We select $\Upsilon$ to ensure that, for $(w, s)\in \{(w, s): 0<|w|<10, |s|\lesssim 1\}\backslash\Upsilon$, at least one $a_{w, s, k}$ is large. We have
\begin{align*}
a_{w, s, 3}&=-(v_1^{\perp})^3\tau_{1, x-z}-(v_2^{\perp})^3\tau_{2, x-z}+(v_1^{\perp})^3\tau_{1, x}+(v_2^{\perp})^3\tau_{2, x}, \\
\frac{1}{3}a_{w, s, 2}&=-(v_1^{\perp})^2(sv_1-w_1)\tau_{1, x-z}-(v_2^{\perp})^2(sv_2-w_2)\tau_{2, x-z}+(v_1^{\perp})^2(sv_1)\tau_{1, x}+(v_2^{\perp})^2(sv_2)\tau_{2, x}, \\
\frac{1}{3}a_{w, s, 1}&=-v_1^{\perp}(sv_1-w_1)^2\tau_{1, x-z}-v_2^{\perp}(sv_2-w_2)^2\tau_{2, x-z}+v_1^{\perp}(sv_1)^2\tau_{1, x}+v_2^{\perp}(sv_2)^2\tau_{2, x}.
\end{align*}
Let $\tau_{j, x-z}\mapsto -\lambda\tau_{j, x-z}'$ and $\tau_{j, x}\mapsto \lambda b_j$ for $j=1, 2$. The above formulas for $a_{w, s, k}$ read as
\begin{equation}  \label{EqContinuousBi-ParameterCoefficients}
\begin{pmatrix} a_{w, s, 3} \\ \frac{1}{3}a_{w, s, 2} \\ \frac{1}{3}a_{w, s, 1} \end{pmatrix}:=\lambda\mathfrak{L}_{w, s}\begin{pmatrix} \tau_{1, x-z}' \\ \tau_{2, x-z}' \\ 1 \end{pmatrix},
\end{equation}
where the linear map $\mathfrak{L}_{w, s}$ is defined by
\begin{equation*}
\mathfrak{L}_{w, s}:=\begin{pmatrix} (v_1^{\perp})^3 & (v_2^{\perp})^3 & b_1(v_1^{\perp})^3+b_2(v_2^{\perp})^3 \\ (v_1^{\perp})^2(sv_1-w_1) & (v_2^{\perp})^2(sv_2-w_2) & b_1(v_1^{\perp})^2 sv_1+b_2(v_2^{\perp})^2 sv_2 \\ v_1^{\perp}(sv_1-w_1)^2 & v_2^{\perp}(sv_2-w_2)^2 & b_1v_1^{\perp}(sv_1)^2+b_2v_2^{\perp}(sv_2)^2 \end{pmatrix}.
\end{equation*}
Take $\Upsilon:=\{(w, s): 0<|w|<10, |s|\lesssim 1, |\det (\mathfrak{L}_{w, s})|\leq \lambda^{-c_1}\}$. Then for $(w, s)\in \{(w, s): 0<|w|<10, |s|\lesssim 1\}\backslash\Upsilon$ we have $\|\mathfrak{L}_{w, s}^{-1}\|\lesssim \lambda^{c_1}$, so that \eqref{EqContinuousBi-ParameterCoefficients} implies
\begin{equation*}
\left| \begin{pmatrix} a_{w, s, 3} \\ \frac{1}{3}a_{w, s, 2} \\ \frac{1}{3}a_{w, s, 1} \end{pmatrix} \right|\gtrsim \lambda \cdot \lambda^{-c_1}=\lambda^{1-c_1}.
\end{equation*}
The desired estimate $|I(w, s)|\lesssim \lambda^{-2c}$ then follows from the van der Corput lemma \cite[Proposition 2.1]{SWPolynomialCarleson}. It remains to check $|\Upsilon|\lesssim \lambda^{-2c}$. Use the polar coordinate system $w\mapsto r(\cos \theta, \sin \theta)$ as before. We have $v=(\cos \theta, \sin \theta), v^{\perp}=(-\sin \theta, \cos \theta)$. Below we calculate $\det (\mathfrak{L}_{w, s})$. Write
\begin{equation}  \label{EqContinuousBi-ParameterMatrix1}
\begin{split}
& \det(\mathfrak{L}_{w, s})=\det \begin{pmatrix} (v_1^{\perp})^3 & (v_2^{\perp})^3 & b_1(v_1^{\perp})^3+b_2(v_2^{\perp})^3 \\ (v_1^{\perp})^2(sv_1-w_1) & (v_2^{\perp})^2(sv_2-w_2) & b_1(v_1^{\perp})^2 sv_1+b_2(v_2^{\perp})^2 sv_2 \\ v_1^{\perp}(sv_1-w_1)^2 & v_2^{\perp}(sv_2-w_2)^2 & b_1v_1^{\perp}(sv_1)^2+b_2v_2^{\perp}(sv_2)^2 \end{pmatrix}  \\
&=\det \begin{pmatrix} -\sin^3 \theta & \cos^3 \theta & -b_1\sin^3 \theta+b_2 \cos^3 \theta \\ (s-r)\sin^2 \theta \cos \theta & (s-r)\cos^2 \theta \sin \theta & s( b_1\sin^2 \theta \cos \theta+b_2\cos^2 \theta \sin \theta) \\ (s-r)^2(-\sin \theta)\cos^2 \theta & (s-r)^2\cos \theta \sin^2 \theta & s^2 \big( b_1(-\sin \theta) \cos^2 \theta+b_2\cos \theta \sin^2 \theta \big) \end{pmatrix}  \\
&=(\cos^2 \theta\sin^2 \theta)\cdot \det \begin{pmatrix} -\sin^3 \theta & \cos^3 \theta & -b_1\sin^3 \theta+b_2 \cos^3 \theta \\ (s-r)\sin \theta  & (s-r)\cos \theta & s( b_1\sin \theta+b_2\cos \theta) \\ (s-r)^2(-\cos \theta) & (s-r)^2 \sin \theta & s^2 \big( b_1(-\cos \theta)+b_2\sin\theta \big) \end{pmatrix} \\
&=(r\cos^2 \theta\sin^2 \theta) \cdot \det \begin{pmatrix} -\sin^3 \theta & \cos^3 \theta & 0 \\ (s-r)\sin \theta  & (s-r)\cos \theta & b_1\sin \theta+b_2\cos \theta \\ (s-r)^2(-\cos \theta) & (s-r)^2 \sin \theta & (2s-r) \big( b_1(-\cos \theta)+b_2\sin\theta \big) \end{pmatrix},
\end{split}
\end{equation}
where in the last equation we subtracted $b_1(\text{the first column})+b_2(\text{the second column})$ from the third column and extracted $r$ outside the matrix. Moreover, we calculate the second term of RHS of \eqref{EqContinuousBi-ParameterMatrix1} as
\begin{equation*} 
\begin{split}
&-(\sin^3 \theta)\det\begin{pmatrix} (s-r)\cos \theta & b_1\sin \theta+b_2\cos \theta \\ (s-r)^2 \sin \theta & (2s-r) \big( b_1(-\cos \theta)+b_2\sin\theta \big) \end{pmatrix} \\
&\ \ \ \ \ \ -(\cos^3 \theta) \det \begin{pmatrix} (s-r)\sin \theta & b_1\sin \theta+b_2\cos \theta \\ (s-r)^2(-\cos \theta) & (2s-r) \big( b_1(-\cos \theta)+b_2\sin\theta \big) \end{pmatrix} \\
&=-(s-r)(\sin^3 \theta) \big( -sb_1 \cos^2 \theta+sb_2\sin \theta \cos \theta-(s-r)b_1 \big) \\
&\ \ \ \ \ \ -(s-r)(\cos^3 \theta) \big( -sb_1 \cos \theta \sin \theta+sb_2\sin^2 \theta+(s-r)b_2 \big) \\
&=-(s-r) \big( -sb_1 \sin\theta \cos^2 \theta+sb_2\cos \theta \sin^2 \theta+(s-r)(b_2\cos^3 \theta-b_1\sin^3 \theta) \big).
\end{split}
\end{equation*}
Plug this into \eqref{EqContinuousBi-ParameterMatrix1}. We get
\begin{equation*}
\det(\mathfrak{L}_{w, s})=-r(s-r)(\cos^2 \theta\sin^2 \theta)\big( -sb_1 \sin\theta \cos^2 \theta+sb_2\cos \theta \sin^2 \theta+(s-r)(b_2\cos^3 \theta-b_1\sin^3 \theta) \big).
\end{equation*}
Thus $\Upsilon$ is contained in the intersection of $\{(w, s): 0<|w|<10, |s|\lesssim 1\}$ and
\begin{equation*}
\begin{split}
\{r\leq \lambda^{-\frac{c_1}{4}}\} & \cup \{|s-r|\leq \lambda^{-\frac{c_1}{4}}\}\cup \{|\cos^2 \theta\sin^2 \theta|\leq \lambda^{-\frac{c_1}{4}}\} \cup \{|b_2\cos^3 \theta-b_1\sin^3 \theta|\leq \lambda^{-\frac{c_1}{8}}\} \\
&\cup \big{\{}\big| \frac{-sb_1 \sin\theta \cos^2 \theta+sb_2\cos \theta \sin^2 \theta}{b_2\cos^3 \theta-b_1\sin^3 \theta}+s-r \big| \leq \lambda^{-\frac{c_1}{8}} \big{\}}.
\end{split}
\end{equation*}
Recall $b_1, b_2$ are constants such that $|(b_1, b_2)|\sim 1$. By Fubini and the elementary inequality $|b_2\cos^3 \theta-b_1\sin^3 \theta|\gtrsim |b_2^{\frac{1}{3}}\cos\theta-b_1^{\frac{1}{3}}\sin\theta|^3$, the intersection of $\{(w, s): 0<|w|<10, |s|\lesssim 1\}$ and each above set has measure $O(\lambda^{-2c})$, so $|\Upsilon|\lesssim \lambda^{-2c}$. This completes the proof of Lemma \ref{LemmaContinuousMajorArcBi-Parameter}.

\section{Finding Patterns in Finite Fields}  \label{SectionFindingPatterns}
In $\mathbb{F}_q^d$, let $\Delta_h f(x)$ denote the standard multiplicative discrete derivative $f(x)\overline{f(x+h)}$. Define $\Delta_{h_1, \cdots, h_n}f:=\Delta_{h_1}\cdots \Delta_{h_n}f$. We will use Gowers norms defined as follows for $n\in \mathbb{N}_{+}$:
\begin{equation*}
\|f\|_{U^n}^{2^n}:=\mathbb{E}_{x, h_1, \cdots, h_n} \Delta_{h_1, \cdots, h_n}f(x)
\end{equation*}
We refer the reader to \cite{TaoHigherOrderFourierAnalysis} and \cite[Appendix A]{PPNonlinearRoth} for basic properties of Gowers norms.

The main goal of this section is to establish an inverse theorem (Theorem \ref{TheoremFindingPatternsInverse} below) for the counting operator
\begin{equation}  \label{EqFindingPatternsCountingOperator}
\Lambda(f_0, f_1, f_2):=\mathbb{E}_{x\in \mathbb{F}_q^4, y\in \mathbb{F}_q^2} f_0(x)f_1 \big( x_{\leq 3}-\Gamma(y), x_4 \big) f_2\big( x_{\leq 3}, x_4-P(y) \big).
\end{equation}
\begin{theorem}  \label{TheoremFindingPatternsInverse}
Let $q$ be a prime and $\delta\in (0, 1)$. Assume $q>C_1\delta^{-C_1}$. Let $f_0, f_1, f_2: \mathbb{F}_q^4\rightarrow \mathbb{C}$ be $1$-bounded functions. Suppose ($\Lambda$ defined by \eqref{EqFindingPatternsCountingOperator})
\begin{equation*}
|\Lambda(f_0, f_1, f_2)|>\delta.
\end{equation*}
Then
\begin{align}
\mathbb{E}_{x_4\in \mathbb{F}_q} |\mathbb{E}_{x_{\leq 3}\in \mathbb{F}_q^3}f_1(\cdot, x_4)|&\gtrsim \delta^{O(1)}, \label{EqFindingPatternsInverse1} \\
\mathbb{E}_{x_{\leq 3}\in \mathbb{F}_q^3} |\mathbb{E}_{x_4\in \mathbb{F}_q}f_2(x_{\leq 3}, \cdot)|&\gtrsim \delta^{O(1)}. \label{EqFindingPatternsInverse2}
\end{align}
\end{theorem}
Choose $\delta:=C_1 q^{-\frac{1}{C_1}}$ so that $q>C_1\delta^{-C_1}$. We can bound $|\Lambda|$ from above:
\begin{corollary}  \label{CorollaryFindingPatternsInverse}
Let $f_0, f_1, f_2: \mathbb{F}_q^4\rightarrow \mathbb{C}$ be $1$-bounded functions. Suppose at least one of the following holds
\begin{enumerate}
\item $\mathbb{E}_{x_4\in \mathbb{F}_q}f_2(x_{\leq 3}, \cdot)=0, \forall x_{\leq 3}\in \mathbb{F}_q^3$.
\item $\mathbb{E}_{x_{\leq 3}\in \mathbb{F}_q^3}f_1(\cdot, x_4)=0, \forall x_4\in \mathbb{F}_q$.
\end{enumerate}
Then we have ($c_0:=\frac{1}{C_1}$)
\begin{equation*}
|\Lambda(f_0, f_1, f_2)|\lesssim q^{-c_0}.
\end{equation*}
\end{corollary}
To see Corollary \ref{CorollaryFindingPatternsInverse} implies Theorem \ref{TheoremFindingPatterns}, we note that LHS of \eqref{EqFindingPatterns} is precisely $\Lambda(1_A, 1_A, 1_A)$ and telescope $\Lambda(1_A, 1_A, 1_A)$ into
\begin{equation*}
\begin{split}
\Lambda \big( 1_A, \mathbb{E}_{x_{\leq 3}} & 1_A(\cdot, x_4), 1_A \big)+\Lambda \big( 1_A, 1_A-\mathbb{E}_{x_{\leq 3}} 1_A(\cdot, x_4), 1_A \big)=\Lambda \big( 1_A, \mathbb{E}_{x_{\leq 3}} 1_A(\cdot, x_4), \mathbb{E}_{x_4} 1_A(x_{\leq 3}, \cdot) \big) \\
&+\Lambda \big( 1_A, \mathbb{E}_{x_{\leq 3}} 1_A(\cdot, x_4), 1_A-\mathbb{E}_{x_4} 1_A(x_{\leq 3}, \cdot) \big)+\Lambda \big( 1_A, 1_A-\mathbb{E}_{x_{\leq 3}} 1_A(\cdot, x_4), 1_A \big).
\end{split}
\end{equation*}
The first term above is the main term in RHS of \eqref{EqFindingPatterns}, and the last two terms are $q^{-c_0}$ by Corollary \ref{CorollaryFindingPatternsInverse}. The lower bound
\begin{equation*}
\Lambda \big( 1_A, \mathbb{E}_{x_{\leq 3}} 1_A(\cdot, x_4), \mathbb{E}_{x_4} 1_A(x_{\leq 3}, \cdot) \big)\geq \alpha^3
\end{equation*}
follows from Sidorenko 3-edge path lemma \cite[Theorem 5.3.4]{ZhaoGraphAndAdditive}.

Below we establish Theorem \ref{TheoremFindingPatternsInverse} following the argument in \cite{PPSSquareCorner} .
\subsection{Gowers norm control}
We first control $|\Lambda|$ by Gowers norms in the $x_4$ direction.
\begin{proposition}  \label{PropositionFindingPatternsGowersNormControl1}
Let $g_0, g_1, g_2: \mathbb{F}_q^3\rightarrow \mathbb{C}$ be $1$-bounded functions. Then there exist $n_0\in \mathbb{N}_+$ and $c_3>0$ such that
\begin{equation*}
\big| \mathbb{E}_{x\in \mathbb{F}_q^3, y\in \mathbb{F}_q} g_0(x)g_1 \big( x_{\leq 2}-(y, y^2), x_3 \big) f_2\big( x_{\leq 2}, x_3-y^3 \big) \big|\leq \big( \mathbb{E}_{x_{\leq 2}\in \mathbb{F}_q^3} \|g_2(x_{\leq 2}, \cdot)\|_{U^{n_0}}^{2^{n_0}} \big)^{\frac{c_3}{2^{n_0}}}+O(q^{-c_3}).
\end{equation*}
\end{proposition}
\begin{proof}
This is a special case of \cite[Proposition 4.1]{KucaPolynomialCornerFiniteFields}. Following the definition there, our special case corresponds to a nice family of polynomials:
\begin{equation*}
\begin{pmatrix} \mathcal{P}_1 \\ \mathcal{P}_2 \\ \mathcal{P}_3 \end{pmatrix}:=\begin{pmatrix} -y & 0 \\ -y^2 & 0 \\ 0 & -y^3 \end{pmatrix}.
\end{equation*}
Hence we can apply \cite[Proposition 4.1]{KucaPolynomialCornerFiniteFields}.
\end{proof}
Recall we have
\begin{equation*}
|\Lambda(f_0, f_1, f_2)|=\big| \mathbb{E}_{x\in \mathbb{F}_q^4, y\in \mathbb{F}_q^2} f_0(x)f_1 \big( x_{\leq 3}-(y_1, y_2, y_1^2+y_2^2), x_4 \big) f_2( x_{\leq 3}, x_4-y_1^3-y_2^3) \big|>\delta.
\end{equation*}
Apply Lemma \ref{LemmaPopularityPrinciple}. We get $>\frac{\delta}{2}q^2$ many pairs $(x_2, y_2)\in \mathbb{F}_q^2$ such that
\begin{equation*}
\big| \mathbb{E}_{x_1, x_3, x_4, y_1} f_0(x)f_1 \big( x_1-y_1, x_2-y_2, (x_3-y_1^2)-y_2^2), x_4 \big) f_2\big( x_{\leq 3}, (x_4-y_1^3)-y_2^3 \big) \big|>\frac{\delta}{2}.
\end{equation*}
Then applying Proposition \ref{PropositionFindingPatternsGowersNormControl1} to each such $(x_2, y_2)$ yields
\begin{equation*}
\mathbb{E}_{x_1, x_3} \|f_2(x_{\leq 3}, \cdot-y_2^3)\|_{U^{n_0}}^{2^{n_0}}>(\frac{\delta}{4})^{O(1)}.
\end{equation*}
Summing over these $(x_2, y_2)$ and using that Gowers norms are translation invariant, we obtain
\begin{equation}  \label{EqFindingPatternsGowersNormControl}
\mathbb{E}_{x_{\leq 3}} \|f_2(x_{\leq 3}, \cdot )\|_{U^{n_0}}^{2^{n_0}}\gtrsim \delta^{O(1)}.
\end{equation}
We have established
\begin{proposition} \label{PropositionFindingPatternsGowersNormControl2}
Let $q$ be a prime and $\delta\in (0, 1)$. Assume $q>C_1\delta^{-C_1}$. Let $f_0, f_1, f_2: \mathbb{F}_q^4\rightarrow \mathbb{C}$ be $1$-bounded functions. Suppose ($\Lambda$ defined by \eqref{EqFindingPatternsCountingOperator})
\begin{equation*}
|\Lambda(f_0, f_1, f_2)|>\delta.
\end{equation*}
Then \eqref{EqFindingPatternsGowersNormControl} holds for some $n_0\in \mathbb{N}_+$.
\end{proposition}

\subsection{Passing to the dual function} \label{SubsectionPassingToTheDualFunction}
The next step is to deduce a Gowers norm control for the dual function of $f_2$ instead of $f_2$ itself. We will work with dual function in the next degree-lowering step as it contains more information about the pattern we are counting.

Given $f_0, f_1$, let
\begin{equation*}
F_2(x):=\mathbb{E}_{y\in \mathbb{F}_q^2} f_0\big( x_{\leq 3}, x_4+P(y) \big)f_1 \big( x_{\leq 3}-\Gamma(y), x_4+P(y) \big),
\end{equation*}
so that
\begin{equation} \label{EqFindingPatternsDualFunctionProperty}
\Lambda(f_0, f_1, f_2)=\langle f_2, F_2 \rangle,\ \forall f_2.
\end{equation}
\begin{proposition}  \label{PropositionFindingPatternsGowersNormControlForDualFunction}
Under the same conditions of Proposition \ref{PropositionFindingPatternsGowersNormControl2}, we have
\begin{equation*}
\mathbb{E}_{x_{\leq 3}} \|F_2(x_{\leq 3}, \cdot )\|_{U^{n_0}}^{2^{n_0}}\gtrsim \delta^{O(1)}
\end{equation*}
for some $n_0\in \mathbb{N}_+$.
\end{proposition}
\begin{proof}
Applying \eqref{EqFindingPatternsDualFunctionProperty} twice together with Cauchy--Schwartz yields
\begin{equation*}
\delta<|\Lambda(f_0, f_1, f_2)|=|\langle f_2, F_2 \rangle|\leq \|f_2\|_2 \cdot \|F_2\|_2\leq \|F_2\|_2=|\langle \overline{F_2}, F_2 \rangle|^{\frac{1}{2}}=|\Lambda(f_0, f_1, \overline{F_2})|^{\frac{1}{2}}.
\end{equation*}
Then apply Proposition \ref{PropositionFindingPatternsGowersNormControl2} with $f_0, f_1, \overline{F_2}$ and $\delta$ replaced by $\delta^2$. The claim follows as Gowers norms are conjugate invariant.
\end{proof}

\subsection{Degree-lowering}
Now we exploits the intuition \footnote{We first learned this intuition from Prof. Prendiville's lecture notes. See \href{https://sites.google.com/view/seanprendiville/}{his homepage} for many relevant lecture notes.} that higher-complexity objects do not contribute seriously to our counting operator \eqref{EqFindingPatternsCountingOperator}. A rigorous version of this intuition is the following:
\begin{proposition}  \label{PropositionFindingPatternsDegree-Lowering}
Let $n\geq 2$. Let $q$ be a prime and $\delta\in (0, 1)$. Assume $q>C(n)\delta^{-C(n)}$. Let $f_0, f_1: \mathbb{F}_q^4\rightarrow \mathbb{C}$ be $1$-bounded functions. Then
\begin{equation*}
\mathbb{E}_{x_{\leq 3}\in \mathbb{F}_q^3} \|F_2(x_{\leq 3}, \cdot )\|_{U^n}^{2^n}>\delta \Longrightarrow \mathbb{E}_{x_{\leq 3}\in \mathbb{F}_q^3} \|F_2(x_{\leq 3}, \cdot )\|_{U^{n-1}}^{2^{n-1}}\gtrsim \delta^{O_n(1)}.
\end{equation*}
\end{proposition}
\begin{remark}
The constants $C(n)$ is chosen to be large depending on $C_0$ in Lemma \ref{LemmaFiniteFieldsMajorArc2}.
\end{remark}
To establish Proposition \ref{PropositionFindingPatternsDegree-Lowering}, we need the following dual-difference interchange lemma and low rank control lemma. Both lemmas are proved in \cite{PPNonlinearRoth} in the integer setting, and the same proofs work in the finite field setting.

\begin{lemma}[{\cite[Lemma 6.3]{PPNonlinearRoth}}]  \label{LemmaFindingPatternsDual-DifferenceInterchange}
Let $Y$ be any index set and $F_y: \mathbb{F}_q\rightarrow \mathbb{C}$ be $1$-bounded functions for $y\in Y$. Set $F(x):=\mathbb{E}_{y\in Y} F_y(x)$. Then for any $\mathcal{H}\subset \mathbb{F}_q^n$ and $\phi: \mathbb{F}_q^n\rightarrow \mathbb{F}_q$ we have
\begin{equation*}
\Big( \mathbb{E}_{h\in \mathcal{H}} \big| \mathbb{E}_{x\in \mathbb{F}_q}\Delta_h F(x)e\big( \phi(h)x \big) \big|   \Big)^{2^n}\lesssim \mathbb{E}_{h^{(0)}, h^{(1)}\in \mathcal{H}} \big| \mathbb{E}_{x\in \mathbb{F}_q} \mathbb{E}_{y\in Y} (\Delta_{h^{(0)}-h^{(1)}} F_y)(x)e\big( \Phi(h^{(0)}, h^{(1)})x \big) \big|.
\end{equation*}
Here $\Phi(h^{(0)}, h^{(1)}):=\sum_{\omega\in \{0,1\}^n} (-1)^{|\omega|} \phi(h^{(\omega)})$ with $h^{(\omega)}:=(h_1^{(\omega_1)}, \cdots, h_n^{(\omega_n)})$.
\end{lemma}

\begin{lemma}[{\cite[Lemma 6.4]{PPNonlinearRoth}}]  \label{LemmaFindingPatternsLowRankControl}
Let $0\leq m\leq n$ and $\phi_1, \cdots, \phi_m: \mathbb{F}_q^{n-1}\rightarrow \mathbb{F}_q$ be arbitrary functions. Let $f: \mathbb{F}_q\rightarrow \mathbb{C}$ be $1$-bounded. Then
\begin{equation*}
\mathbb{E}_{h_1, \cdots, h_n\in \mathbb{F}_q} \Big| \mathbb{E}_{x\in \mathbb{F}_q} \Delta_h f(x)e \Big( \big( \sum_{j=1}^m \phi_j(h_1, \cdots, h_{j-1}, h_{j+1}, \cdots, h_n) \big)x \Big) \Big|\lesssim \|f\|_{U^{n+1}}^{2^{n-m}}.
\end{equation*}
\end{lemma}

\begin{proof}[Proof of Proposition \ref{PropositionFindingPatternsDegree-Lowering}]
First assume $n=2$ which is the easiest case. Our condition reads as
\begin{equation*}
\mathbb{E}_{x_{\leq 3}\in \mathbb{F}_q^3} \|F_2(x_{\leq 3}, \cdot )\|_{U^2}^{4}>\delta.
\end{equation*}
By Lemma \ref{LemmaPopularityPrinciple}, we can find a set $\Upsilon\subset \mathbb{F}_q^3$ with $\#\Upsilon>\frac{\delta}{2}q^3$ such that
\begin{equation*}
\|F_2(x_{\leq 3}, \cdot )\|_{U^2}^{4}>\frac{\delta}{2},\ \forall x_{\leq 3}\in \Upsilon.
\end{equation*}
Recall that the $U^2$ inverse theorem $\|g\|_{U^2}\leq \|\hat{g}\|_{\infty}^{\frac{1}{2}}$ holds for any 1-bounded function $g$. Apply this to each $x_{\leq 3}\in \Upsilon$ and find $\phi_{x_{\leq 3}}\in \mathbb{F}_q$ which maximizes $\|\widehat{F_2(x_{\leq 3}, \cdot )}\|_{\infty}$. We get
\begin{equation} \label{EqFindingPatternsDegree-Lowering1}
|\mathbb{E}_{x_4\in \mathbb{F}_q} F_2(x_{\leq 3}, x_4) e(-\phi_{x_{\leq 3}}\cdot x_4)|>(\frac{\delta}{2})^{\frac{1}{2}},\ \forall x_{\leq 3}\in \Upsilon.
\end{equation}
Summing \eqref{EqFindingPatternsDegree-Lowering1} over $x_{\leq 3}\in \Upsilon$ and expanding the definition of $F_2$, we obtain
\begin{equation*}   
\mathbb{E}_{x_{\leq 3}\in \Upsilon} \big| \mathbb{E}_{x_4\in \mathbb{F}_q} \mathbb{E}_{y\in \mathbb{F}_q^2} f_0\big( x_{\leq 3}, x_4+P(y) \big)f_1 \big( x_{\leq 3}-\Gamma(y), x_4+P(y) \big) e(-\phi_{x_{\leq 3}}\cdot x_4)  \big|\gtrsim \delta^{O(1)}.
\end{equation*}
Let $x_4\mapsto x_4-P(y)$, take $\mathbb{E}_{x_4\in \mathbb{F}_q}$ outside the absolute value, and then pigeonhole on $x_4$. We get one $x_4\in \mathbb{F}_q$ such that
\begin{equation*}
\mathbb{E}_{x_{\leq 3}\in \Upsilon} \big| \mathbb{E}_{y\in \mathbb{F}_q^2} f_1 \big( x_{\leq 3}-\Gamma(y), x_4 \big) e \big( \phi_{x_{\leq 3}}\cdot P(y) \big)  \big|\gtrsim \delta^{O(1)}.
\end{equation*}
We apply Lemma \ref{LemmaFiniteFieldsMajorArc2} to find $\Upsilon_0\subset \Upsilon$ with $\Upsilon_0\gtrsim \delta^{O(1)}q^3$ such that $\phi_{x_{\leq 3}}=0, \forall x_{\leq 3}\in \Upsilon_0$. Now go back to \eqref{EqFindingPatternsDegree-Lowering1}. Summing \eqref{EqFindingPatternsDegree-Lowering1} over $x_{\leq 3}\in \Upsilon_0$ and extending the sum to all $x_{\leq 3}\in \mathbb{F}_q^3$ by positivity, we obtain
\begin{equation*}
\mathbb{E}_{x_{\leq 3}\in \mathbb{F}_q^3} |\mathbb{E}_{x_4\in \mathbb{F}_q} F_2(x_{\leq 3}, x_4) \big| \gtrsim \delta^{O(1)}.
\end{equation*}
This gives the conclusion by Cauchy--Schwartz.

Then we handle the $n\geq 3$ case. Write our condition as
\begin{equation*}
\mathbb{E}_{x_{\leq 3}\in \mathbb{F}_q^3} \mathbb{E}_{h\in \mathbb{F}_q^{n-2}}  \|(\Delta_h F_2)(x_{\leq 3}, \cdot )\|_{U^2}^{4}>\delta.
\end{equation*}
By Lemma \ref{LemmaPopularityPrinciple}, we can find a set $\Upsilon\subset \mathbb{F}_q^3\times \mathbb{F}_q^{n-2}$ with $\#\Upsilon>\frac{\delta}{2}q^{n+1}$ such that
\begin{equation*}
\|(\Delta_h F_2)(x_{\leq 3}, \cdot )\|_{U^2}^{4}>\frac{\delta}{2},\ \forall (x_{\leq 3}, h)\in \Upsilon.
\end{equation*}
We shall use the following fibers of $\Upsilon$:
\begin{equation*}
\mathcal{H}(x_{\leq 3}):=\{h\in \mathbb{F}_q^{n-2}: (x_{\leq 3}, h)\in \Upsilon \},\ \ \ X(h):=\{x_{\leq 3}\in \mathbb{F}_q^3: (x_{\leq 3}, h)\in \Upsilon \}.
\end{equation*}
For each $(x_{\leq 3}, h)\in \Upsilon$, apply the $U^2$ inverse theorem and find $\phi_{x_{\leq 3}}(h)\in \mathbb{F}_q$ which maximizes $\|\widehat{(\Delta_h F_2)(x_{\leq 3}, \cdot )}\|_{\infty}$. We get
\begin{equation} \label{EqFindingPatternsDegree-Lowering2}
\big| \mathbb{E}_{x_4\in \mathbb{F}_q} (\Delta_h F_2)(x_{\leq 3}, x_4) e \big( -\phi_{x_{\leq 3}}(h)\cdot x_4 \big) \big| >(\frac{\delta}{2})^{\frac{1}{2}},\ \forall (x_{\leq 3}, h)\in \Upsilon.
\end{equation}
Extend $\phi_{x_{\leq 3}}(h)$ to the whole $\mathbb{F}_q^3\times \mathbb{F}_q^{n-2}$ by arbitrarily defining $\phi_{x_{\leq 3}}(h)$ for $(x_{\leq 3}, h)\notin \Upsilon$. Sum \eqref{EqFindingPatternsDegree-Lowering2} over $(x_{\leq 3}, h)\in \Upsilon$:
\begin{equation*}
\mathbb{E}_{x_{\leq 3}\in \mathbb{F}_q^3}\mathbb{E}_{h\in \mathcal{H}(x_{\leq 3})} \Big( \big| \mathbb{E}_{x_4\in \mathbb{F}_q} (\Delta_h F_2)(x_{\leq 3}, x_4) e \big( -\phi_{x_{\leq 3}}(h)\cdot x_4 \big) \big| \Big)^{2^{n-2}} \gtrsim \delta^{O(1)}.
\end{equation*}
Then applying Lemma \ref{LemmaFindingPatternsDual-DifferenceInterchange} yields
\begin{equation*}
\begin{split}
\mathbb{E}_{x_{\leq 3}\in \mathbb{F}_q^3}\mathbb{E}_{h^{(0)}, h^{(1)}\in \mathcal{H}(x_{\leq 3})} \big| & \mathbb{E}_{x_4\in \mathbb{F}_q} \mathbb{E}_{y\in \mathbb{F}_q^2} (\Delta_{h^{(0)}-h^{(1)}} f_0) \big( x_{\leq 3}, x_4+P(y) \big) (\Delta_{h^{(0)}-h^{(1)}} f_1) \\
& \big( x_{\leq 3}-\Gamma(y), x_4+P(y) \big) e \big( -\Phi_{x_{\leq 3}}(h^{(0)}, h^{(1)})\cdot x_4 \big) \big|\gtrsim \delta^{O(1)}.
\end{split}
\end{equation*}
Control the left hand side by $\mathbb{E}_{h^{(1)}\in \mathbb{F}_q^{n-2}}\mathbb{E}_{x_{\leq 3}\in \mathbb{F}_q^3}\mathbb{E}_{h^{(0)}\in \mathcal{H}(x_{\leq 3})} |\cdots|$. By pigeonholing, we can fix one $h^{(1)}\in \mathbb{F}_q^{n-2}$ such that
\begin{equation*}
\begin{split}
\mathbb{E}_{x_{\leq 3}\in \mathbb{F}_q^3}\mathbb{E}_{h^{(0)}\in \mathcal{H}(x_{\leq 3})} \big| & \mathbb{E}_{x_4\in \mathbb{F}_q} \mathbb{E}_{y\in \mathbb{F}_q^2} (\Delta_{h^{(0)}-h^{(1)}} f_0) \big( x_{\leq 3}, x_4+P(y) \big) (\Delta_{h^{(0)}-h^{(1)}} f_1) \\
& \big( x_{\leq 3}-\Gamma(y), x_4+P(y) \big) e \big( -\Phi_{x_{\leq 3}}(h^{(0)}, h^{(1)})\cdot x_4 \big) \big| \gtrsim \delta^{O(1)}.
\end{split}
\end{equation*}
Let $x_4\mapsto x_4-P(y)$, take $\mathbb{E}_{x_4\in \mathbb{F}_q}$ outside the absolute value, and then pigeonhole on $x_4$. We get one $x_4\in \mathbb{F}_q$ such that
\begin{equation*}
\mathbb{E}_{(x_{\leq 3}, h^{(0)})\in \Upsilon} \big| \mathbb{E}_{y\in \mathbb{F}_q^2} (\Delta_{h^{(0)}-h^{(1)}} f_1)\big( x_{\leq 3}-\Gamma(y), x_4 \big) e \big( \Phi_{x_{\leq 3}}(h^{(0)}, h^{(1)})\cdot P(y) \big)  \big|\gtrsim \delta^{O(1)}.
\end{equation*}
Write $\mathbb{E}_{(x_{\leq 3}, h^{(0)})\in \Upsilon}=\mathbb{E}_{h^{(0)}\in \mathbb{F}_q^{n-2}}\mathbb{E}_{x_{\leq 3}\in X(h^{(0)})}$. Applying Lemma \ref{LemmaPopularityPrinciple} gives $\gtrsim \delta^{O((1)}q^{n-2}$ many $h^{(0)}$ such that
\begin{equation*}
\mathbb{E}_{x_{\leq 3}\in X(h^{(0)})} \big| \mathbb{E}_{y\in \mathbb{F}_q^2} (\Delta_{h^{(0)}-h^{(1)}} f_1)\big( x_{\leq 3}-\Gamma(y), x_4 \big) e \big( \Phi_{x_{\leq 3}}(h^{(0)}, h^{(1)})\cdot P(y) \big) \big|\gtrsim \delta^{O(1)}.
\end{equation*}
We apply Lemma \ref{LemmaFiniteFieldsMajorArc2} to find $X_0(h^{(0)})\subset X(h^{(0)})$ with $X_0(h^{(0)})\gtrsim \delta^{O(1)}q^3$ and $\Phi_{x_{\leq 3}}(h^{(0)}, h^{(1)})=0, \forall x_{\leq 3}\in X_0(h^{(0)})$. Altogether, we get $\gtrsim \delta^{O((1)}q^{n+1}$ many pairs $(x_{\leq 3}, h^{(0)})\in \Upsilon$ such that $\Phi_{x_{\leq 3}}(h^{(0)}, h^{(1)})=0$. Let $\Upsilon'$ denote the set of these pairs. For each $(x_{\leq 3}, h^{(0)})\in \Upsilon'$, expanding the definition of $\Phi_{x_{\leq 3}}(h^{(0)}, h^{(1)})$ yields
\begin{equation}  \label{EqFindingPatternsDegree-Lowering3}
\phi_{x_{\leq 3}}(h^{(0)})=-\sum_{\omega\in \{0,1\}^n\backslash \{(0, \cdots, 0)\}} (-1)^{|\omega|} \phi_{x_{\leq 3}}(h^{(\omega)}),\ \forall (x_{\leq 3}, h^{(0)})\in \Upsilon'. 
\end{equation}
Let $\psi_{x_{\leq 3}}(h^{(0)})$ denote RHS of \eqref{EqFindingPatternsDegree-Lowering3} for any $h^{(0)}\in \mathbb{F}_q^{n-2}$. Recall that $h^{(1)}$ is fixed, so $\psi_{x_{\leq 3}}(h^{(0)})$ has rank lower than $n-2$ in the sense that each summand in RHS of \eqref{EqFindingPatternsDegree-Lowering3} depends on at most $n-3$ coordinates of $h^{(0)}$. Now go back to \eqref{EqFindingPatternsDegree-Lowering2}. Since $\#\Upsilon'\gtrsim \delta^{O(1)}q^{n+1}$, Lemma \ref{LemmaPopularityPrinciple} implies that there are $\gtrsim \delta^{O(1)}q^{3}$ many $x_{\leq 3}$ such that $\mathcal{H}'(x_{\leq 3}):=\{h\in \mathbb{F}_q^{n-2}: (x_{\leq 3}, h)\in \Upsilon'\}$ has cardinality $\gtrsim \delta^{O(1)}q^{n-2}$. For each such $x_{\leq 3}$, sum \eqref{EqFindingPatternsDegree-Lowering2} over $h\in \mathcal{H}'(x_{\leq 3})$ and extend to all $h\in \mathbb{F}_q^{n-2}$ by positivity. We obtain ($\phi_{x_{\leq 3}}(h)=\psi_{x_{\leq 3}}(h)$ for $h\in \mathcal{H}'(x_{\leq 3})$)
\begin{equation*}
\mathbb{E}_{h\in \mathbb{F}_q^{n-2}}\big| \mathbb{E}_{x_4\in \mathbb{F}_q} (\Delta_h F_2)(x_{\leq 3}, x_4) e \big( -\psi_{x_{\leq 3}}(h)\cdot x_4 \big) \big|
\end{equation*}
Applying Lemma \ref{LemmaFindingPatternsLowRankControl} with $m, n$ there both replaced by $n-2$, we get $\|F_2(x_{\leq 3}, \cdot)\|_{U^{n-1}}\gtrsim \delta^{O(1)}$. This holds for $\gtrsim \delta^{O(1)}q^{3}$ many $x_{\leq 3}$. Summing these $x_{\leq 3}$ gives the desired conclusion.
\end{proof}

\subsection{Using information for shorter progressions}
In this subsection, we conclude the proof of Theorem \ref{TheoremFindingPatternsInverse} using the counting results for shorter progressions below:
\begin{lemma}  \label{LemmaFindingPatternsShorterProgressions}
Let $q$ be a prime. Then the following hold
\begin{enumerate}
\item For 1-bounded functions $g_0, g_1: \mathbb{F}_q^3\rightarrow \mathbb{C}$ we have
\begin{equation*}
\mathbb{E}_{x\in \mathbb{F}_q^3, y\in \mathbb{F}_q^2} g_0(x)g_1\big( x-\Gamma(y) \big)=(\mathbb{E}_{x\in \mathbb{F}_q^3}g_0)(\mathbb{E}_{x\in \mathbb{F}_q^3}g_1)+O(q^{-c}).
\end{equation*}
\item For 1-bounded functions $g_0, g_1: \mathbb{F}_q\rightarrow \mathbb{C}$ we have
\begin{equation*}
\mathbb{E}_{x\in \mathbb{F}_q, y\in \mathbb{F}_q^2} g_0(x)g_1\big( x-P(y) \big)=(\mathbb{E}_{x\in \mathbb{F}_q}g_0)(\mathbb{E}_{x\in \mathbb{F}_q}g_1)+O(q^{-c}).
\end{equation*}
\end{enumerate}
\end{lemma}
\begin{proof}
The two results are proved in the same way, so we only prove the first. Write
\begin{equation*}
\begin{split}
\mathbb{E}_{x\in \mathbb{F}_q^3, y\in \mathbb{F}_q^2} g_0(x)g_1\big( x-\Gamma(y) \big)&=\mathbb{E}_{x\in \mathbb{F}_q^3, y\in \mathbb{F}_q^2} \sum_{\xi, \tau\in \mathbb{F}_q^3}  e(\xi\cdot x) \widehat{g_0}(\xi) e \big( \tau\cdot (x-\Gamma(y) \big) \widehat{g_1}(\tau) \\
&=\sum_{\xi\in \mathbb{F}_q^3} \widehat{g_0}(\xi) \widehat{g_1}(-\xi) \Big( \mathbb{E}_{y\in \mathbb{F}_q^2} e\big( \xi\cdot \Gamma(y) \big) \Big) \\
&=\widehat{g_0}(0) \widehat{g_1}(0)+O \Big( \sup_{\xi\neq 0} \big| \mathbb{E}_{y\in \mathbb{F}_q^2} e \big( \xi\cdot \Gamma(y) \big) \big|\cdot \|g_0\|_2\cdot \|g_0\|_1 \Big),
\end{split}
\end{equation*}
where we used Cauchy--Schwartz and Plancherel in the last equation. The first term on RHS is exactly $(\mathbb{E}_{x\in \mathbb{F}_q^3}g_0)(\mathbb{E}_{x\in \mathbb{F}_q^3}g_1)$, and the second term is $O(q^{-c})$ by Lemma \ref{LemmaExponentialSumFiniteFields}.
\end{proof}

To show \eqref{EqFindingPatternsInverse1}, we first apply Proposition \ref{PropositionFindingPatternsGowersNormControlForDualFunction}. Note that the constant $C_1$ in Theorem \ref{TheoremFindingPatternsInverse} is chosen to be sufficiently large depending on $n_0$ and $C(n_0), C(n_0-1), \cdots, C(2)$ in Proposition \ref{PropositionFindingPatternsDegree-Lowering}. Repeatedly using Proposition \ref{PropositionFindingPatternsDegree-Lowering} for $n=n_0, n_0-1, \cdots, 2$ yields
\begin{equation*}
\mathbb{E}_{x_{\leq 3}\in \mathbb{F}_q^3} |\mathbb{E}_{x_4\in \mathbb{F}_q} F_2(x_{\leq 3}, x_4)|^2\gtrsim \delta^{O(1)}.
\end{equation*}
We can drop the square as the summands are 1-bounded. Expand the definition of $F_2$. Let $x_4\mapsto x_4-P(y)$, and remove the absolute value by absorbing the sign into $f_0$. We obtain
\begin{equation*}
\big| \mathbb{E}_{x_{\leq 3}\in \mathbb{F}_q^3} \mathbb{E}_{x_4\in \mathbb{F}_q} \mathbb{E}_{y\in \mathbb{F}_q^2} g_0(x_{\leq 3}, x_4)f_1(x_{\leq 3}-\Gamma(y), x_4) \big| \gtrsim \delta^{O(1)}
\end{equation*}
for some 1-bounded function $g_0$. This reads as
\begin{equation*}
\mathbb{E}_{x_4\in \mathbb{F}_q} \big| \mathbb{E}_{x_{\leq 3}} \mathbb{E}_{y\in \mathbb{F}_q^2} g_0(x_{\leq 3}, x_4)f_1(x_{\leq 3}-\Gamma(y), x_4) \big|\gtrsim \delta^{O(1)}.
\end{equation*}
For each $x_4\in \mathbb{F}_q$, apply the fist conclusion of Lemma \ref{LemmaFindingPatternsShorterProgressions} and use $|\mathbb{E}_{x_{\leq 3}} g_0(\cdot, x_4)|\leq 1$. We get \eqref{EqFindingPatternsInverse1} as desired.

To show \eqref{EqFindingPatternsInverse2}, we pass the the dual function of $f_1$:
\begin{equation*}
F_1(x):=\mathbb{E}_{y\in \mathbb{F}_q^2} f_0\big( x_{\leq 3}+\Gamma (y), x_4\big)f_2 \big( x_{\leq 3}+\Gamma(y), x_4-P(y) \big).
\end{equation*}
We have
\begin{equation*}
\Lambda(f_0, f_1, f_2)=\langle f_1, F_1 \rangle,\ \forall f_1.
\end{equation*}
A similar argument as in Subsection \ref{SubsectionPassingToTheDualFunction} then gives
\begin{equation*}
|\Lambda(f_0, \overline{F_1}, f_2)|>\delta^2.
\end{equation*}
Applying \eqref{EqFindingPatternsInverse1} with functions $f_0, \overline{F_1}, f_2$ yields
\begin{equation*}
\mathbb{E}_{x_4\in \mathbb{F}_q} \big| \mathbb{E}_{x_{\leq 3}} \mathbb{E}_{y\in \mathbb{F}_q^2} f_0\big( x_{\leq 3}+\Gamma (y), x_4\big)f_2 \big( x_{\leq 3}+\Gamma(y), x_4-P(y) \big) \big|\gtrsim \delta^{O(1)}.
\end{equation*}
Let $x_{\leq}\mapsto x_{\leq 3}-\Gamma(y)$, remove the absolute value by absorbing the sign into $f_0$, and apply the second conclusion of Lemma \ref{LemmaFindingPatternsShorterProgressions} fiberwisely. We get \eqref{EqFindingPatternsInverse2} and complete the proof of Theorem \ref{TheoremFindingPatternsInverse}.

\section{A Pierce--Yung Operator Revisited} \label{SectionPY}
Here we give a new proof of Theorem \ref{TheoremPY} using Lemma \ref{LemmaContinuousMajorArc}.

Let $d\in \mathbb{N}_+$. Fix $\phi\in \mathcal{S}(\mathbb{R}^d)$ such that $\overline{B(0, 1)}\prec \phi \prec B(0, 2)$ and $\psi(\cdot):=\phi(\cdot)-\phi(2\cdot)$. Define $\phi_s:=\phi(\frac{\cdot}{2^s}), \psi_s:=\psi(\frac{\cdot}{2^s})$ for $s\in \mathbb{Z}$. Hence $\psi_s$ is supported in the annulus $\{y: \frac{1}{2}\cdot 2^s<|y|<2\cdot 2^s\}$, and we have
\begin{equation*}
\sum_{s\in \mathbb{Z}} \psi_{s}(y)=1, \ \ \ \forall y\neq 0.
\end{equation*}
These are standard Littlewood--Paley projections. We shall use them in the physical space with $d=2$ and in the frequency space with $d=2, 1$.

By the standard linearizing argument, it suffices to show the following operator is $L^p$ bounded for any fixed measurable function $\tau_{(\cdot)}: \mathbb{R}^3\rightarrow \mathbb{R}$:
\begin{equation*}
\int_{\mathbb{R}^2} f \big( x-\Gamma(y) \big) e^{2\pi i \tau_x P(y)} K(y) dy.
\end{equation*}
We decompose
\begin{equation*}
\int_{\mathbb{R}^2} f \big( x-\Gamma(y) \big) e^{2\pi i \tau_x P(y)} K(y)=\sum_{s\in \mathbb{Z}} \int_{\mathbb{R}^2} f \big( x-\Gamma(y) \big) e^{2\pi i \tau_x P(y)} \psi_s(y)K(y).
\end{equation*}
Let $K_s(y):=\psi_s(y)K(y)$. By approximation, it suffices to prove $L^p$ boundedness of the following operator uniformly in $N$:
\begin{equation*}
\sum_{|s|\leq N} \int_{\mathbb{R}^2} f \big( x-\Gamma(y) \big) e^{2\pi i \tau_x P(y)} K_s(y).
\end{equation*}
Fix $N$ from now on. We will use boundedness of Radon transforms.
\begin{theorem}[{\cite[Appendix A]{MSTDiscreteRadon1}}] \label{TheoremRadonTransforms}
Define the maximal Radon transforms $M^{\Gamma}$ and $M^{\Gamma, K}$ by
\begin{align*}
M^{\Gamma}f(x) & :=\sup_{R>0} \frac{1}{R^2}\int_{|y|\leq R} \big| f \big( x-\Gamma(y) \big) \big| dy, \\
M^{\Gamma, K}f(x) & :=\sup_{R>0} \big| \int_{|y|\leq R} f \big( x-\Gamma(y) \big) K(y) dy \big|,
\end{align*}
respectively. Then for $1<p<\infty$ and $f\in \mathcal{S}(\mathbb{R}^3)$ we have
\begin{equation*}
\|M^{\Gamma}f\|_p\lesssim \|f\|_p,\ \ \ \|M^{\Gamma, K}f\|_p\lesssim \|f\|_p.
\end{equation*}
\end{theorem}

As in \cite{SWPolynomialCarleson, PYPolynomialCarlesonRadon}, we first decompose our operator into oscillatory and non-oscillatory parts:
\begin{equation}  \label{EqPYOscillatoryAndNon-Oscillatory}
\sum_{|s|\leq N} \int_{\mathbb{R}^2} f \big( x-\Gamma(y) \big) e^{2\pi i \tau_x P(y)} K_s(y)=\sum_{|s|\leq N} 1_{|\tau_x|>2^{-3s}}\int_{\mathbb{R}^2} \cdots+\sum_{|s|\leq N} 1_{|\tau_x|\leq 2^{-3s}}\int_{\mathbb{R}^2} \cdots.
\end{equation}
For the second part (i.e. non-oscillatory part) in RHS of \eqref{EqPYOscillatoryAndNon-Oscillatory}, we use
\begin{equation*}
\begin{split}
\big| \int_{\mathbb{R}^2} f \big( x-\Gamma(y) \big) e^{2\pi i \tau_x P(y)} K_s(y)-\int_{\mathbb{R}^2} f \big( x-\Gamma(y) \big) K_s(y) \big| & \lesssim O(|\tau_x| 2^{3s})\int_{|y|\lesssim 2^s} |f \big( x-\Gamma(y) \big)| \frac{1}{2^{2s}} \\
&\lesssim O(|\tau_x| 2^{3s}) M^{\Gamma}f(x),
\end{split}
\end{equation*}
and
\begin{equation*}
\big| \sum_{s_1\leq s\leq s_2}\int_{\mathbb{R}^2} f \big( x-\Gamma(y) \big) K_s(y)-\int_{2^{s_1}<|y|\leq 2^{s_2}} f \big( x-\Gamma(y) \big) K(y) \big|\lesssim M^{\Gamma}f(x),
\end{equation*}
to deduce that
\begin{equation*}
\begin{split}
|\text{second part of \eqref{EqPYOscillatoryAndNon-Oscillatory}}| & \leq \big| \sum_{|s|\leq N: |\tau_x|\leq 2^{-3s}} \int_{\mathbb{R}^2} f \big( x-\Gamma(y) \big) K_s(y) \big|+\sum_{|s|\leq N: |\tau_x|\leq 2^{-3s}} O(|\tau_x| 2^{3s}) M^{\Gamma}f(x)  \\
& \lesssim \big( M^{\Gamma, K}f(x)+M^{\Gamma}f(x) \big)+M^{\Gamma}f(x)\lesssim M^{\Gamma, K}f(x)+M^{\Gamma}f(x).
\end{split}
\end{equation*}
Thus $\|\text{second part of \eqref{EqPYOscillatoryAndNon-Oscillatory}}\|_p\lesssim \|f\|_p$ by Theorem \ref{TheoremRadonTransforms}.

For the first part (i.e. oscillatory part) in RHS of \eqref{EqPYOscillatoryAndNon-Oscillatory}, we further decompose
\begin{equation*}
\begin{split}
\text{first part of \eqref{EqPYOscillatoryAndNon-Oscillatory}} & =\sum_{|s|\leq N} \sum_{l\geq 1} 1_{|\tau_x|\sim 2^{-3s+l}}\int_{\mathbb{R}^2} f \big( x-\Gamma(y) \big) e^{2\pi i \tau_x P(y)} K_s(y) \\
&=\sum_{l\geq 1} \Big( \sum_{|s|\leq N} 1_{|\tau_x|\sim 2^{-3s+l}}\int_{\mathbb{R}^2} f \big( x-\Gamma(y) \big) e^{2\pi i \tau_x P(y)} K_s(y) \Big) \\
&=:\sum_{l\geq 1} T_l f(x)
\end{split}
\end{equation*}
Then we see Theorem \ref{TheoremPY} is implied by Proposition \ref{PropositionPYMaximalMajorArcWeaker} below.
\begin{proposition}  \label{PropositionPYMaximalMajorArcWeaker}
For $1<p<\infty$ and $l\geq 1$, we have
\begin{equation*}
\|T_l f\|_p\lesssim 2^{-c(p)l}\|f\|_p.
\end{equation*}
\end{proposition}
\begin{remark}
We will use the fact that $K_s$ is supported in $\{y: \frac{1}{2}\cdot 2^s<|y|<2\cdot 2^s\}$, which makes Proposition \ref{PropositionPYMaximalMajorArcWeaker} weaker than Lemma \ref{LemmaMaximalMajorArc}.
\end{remark}
\begin{proof}
Note that we have
\begin{equation*}
|T_l f(x)| \lesssim M^{\Gamma}f(x),
\end{equation*}
so $\|T_l f\|_p\lesssim \|f\|_p$ for $1<p<\infty$. By interpolation, it suffices to show $\|T_l f\|_2\lesssim 2^{-cl}\|f\|_2$. We can focus on the $l\gtrsim 1$ case. Moreover, we can assume $\hat{f}$ is supported in a compact set in $\mathbb{R}^3\backslash \{\xi_{\leq 2}=0\text{ or } \xi_3=0\}$. (This is not strictly necessary, but it makes the later argument cleaner.) 

The argument is based on a two-step paraproduct decomposition. For each $|s|\leq N$, we decompose
\begin{equation*}
1_{\mathbb{R}^3\backslash \{\xi_{\leq 2}=0\text{ or } \xi_3=0\}}=\sum_{k, k'\in \mathbb{Z}^2} \psi_{-s+k}(\xi_{\leq 2})\psi_{-2s+k'} (\xi_3).
\end{equation*}
Here the first $\psi$ is a two-dimensional Littlewood--Paley projection, and the second is one-dimensional. Plug this into $T_l$. We write $T_l f$ as
\begin{equation}  \label{EqPYMainDecomposition}
\sum_{|s|\leq N}\sum_{(k, k')\in \mathbb{Z}^2} 1_{|\tau_x|\sim 2^{-3s+l}}\int_{\mathbb{R}^2} \big( \psi_{-s+k}(D_{\leq 2}) \otimes \psi_{-2s+k'} (D_3)f \big) \big( x-\Gamma(y) \big) e^{2\pi i \tau_x P(y)} K_s(y).
\end{equation}
Note that our assumption on $f$ implies there are only finitely many non-zero summands above, so we can freely reorganize the sum. For any $\Upsilon\subset \mathbb{Z}^2$, we define $T_l^{\Upsilon} f$ to be the operator in \eqref{EqPYMainDecomposition} with $\sum_{(k, k')\in \mathbb{Z}^2}$ replaced by $\sum_{(k, k')\in \Upsilon}$. Fix a large absolute constant $C$. We partition $\mathbb{Z}^2$ into the following four subsets
\begin{align*}
\Upsilon_{<}&:=\big{\{} (k, k')\in \mathbb{Z}^2: \max\{k, k'\}\leq l-C \big{\}}, \\
\Upsilon_{>, \neq}&:=\big{\{} (k, k')\in \mathbb{Z}^2: \max\{k, k'\}\geq l+C, |k-k'|\geq C \big{\}}, \\
\Upsilon_{>, =}&:=\big{\{} (k, k')\in \mathbb{Z}^2: \max\{k, k'\}\geq l+C, |k-k'|<C \big{\}}, \\
\Upsilon_{=}&:=\big{\{} (k, k')\in \mathbb{Z}^2: l-C<\max\{k, k'\}< l+C \big{\}}.
\end{align*}
Decompose $T_l f=T_l^{\Upsilon_{<}}f+T_l^{\Upsilon_{>, \neq}}f+T_l^{\Upsilon_{>, =}}f+T_l^{\Upsilon_{=}}f$ accordingly. We shall prove $\|T_l^{\Upsilon}f\|_2\lesssim 2^{-cl}\|f\|_2$ for $\Upsilon=\Upsilon_{<}, \Upsilon_{>, \neq}, \Upsilon_{>, =}, \Upsilon_{=}$ separately. The off-diagonal parts $T_l^{\Upsilon_{<}}f, T_l^{\Upsilon_{>, \neq}}f$ have rapid decay, and we will control them pointwisely by the strong maximal operator. The two types of diagonal parts $T_l^{\Upsilon_{>, =}}f, T_l^{\Upsilon_{=}}f$ have frequency orthogonality for slightly different reasons, and we will sum them together using the two types of orthogonality.

Now we go to the details. Recall the strong maximal operator defined by
\begin{equation*}
M^{\text{str}} f(x):=\sup_{R_i>0, i\in [3]} \frac{1}{R_1 R_2 R_3}\int_{|y_i|\leq R_i, i\in [3]} |f(x-y)|.
\end{equation*}
One can control $M^{\text{str}} f$ by $(M_1 \circ M_2 \circ M_3)f$, where $M_i$ is the one-dimensional Hardy--Littlewood maximal operator acting on the $i\text{-th}$ coordinate only. Hence $\|M^{\text{str}} f\|_p\lesssim \|f\|_p$ for $1<p<\infty$.

For $T_l^{\Upsilon_{<}}f$, it suffices to show
\begin{equation}  \label{EqPYOperatorPartOne}
|T_l^{\Upsilon_{<}}f(x)|\lesssim 2^{-cl} M^{\text{str}} f(x),\ \forall x\in \mathbb{R}^3.
\end{equation}
Fix $x$ and abbreviate $\tau_x$ as $\tau$. There is at most one $s$ such that $|\tau|\sim 2^{-3s+l}$. Assume such an $s$ exists. Fix it as well. It suffices to show
\begin{equation}  \label{EqPYOperatorPartOne1}
\big| \sum_{(k, k')\in \Upsilon_{<}} \int_{\mathbb{R}^2} \big( \psi_{-s+k}(D_{\leq 2}) \otimes \psi_{-2s+k'} (D_3)f \big) \big( x-\Gamma(y) \big) e^{2\pi i \tau P(y)} K_s(y) \big| \lesssim 2^{-cl} M^{\text{str}} f(x).
\end{equation}
Write
\begin{equation} \label{EqPYOperatorPartOne2}
\text{LHS of \eqref{EqPYOperatorPartOne1}}=\big| \int_{\mathbb{R}^2} \big( \phi_{-s+l-C}(D_{\leq 2}) \otimes \phi_{-2s+l-C} (D_3)f \big) \big( x-\Gamma(y) \big) e^{2\pi i \tau_x P(y)} K_s(y) \big|.
\end{equation}
Note that $\big( \phi_{-s+l-C}(D_{\leq 2})\otimes \phi_{-2s+l-C} (D_3)f \big) \big( x-\Gamma(y) \big)$ is equal to
\begin{equation*}
\begin{split}
\int_{\mathbb{R}^3} f\big( x-\Gamma(y)-t \big)& (\phi_{-s+l-C}\otimes \phi_{-2s+l-C})^{\vee}(t) dt \\
&=\int_{\mathbb{R}^3} f(x-t) (\phi_{-s+l-C}\otimes \phi_{-2s+l-C})^{\vee} \big( t-\Gamma(y) \big) dt.
\end{split}
\end{equation*}
Plug this into RHS of \eqref{EqPYOperatorPartOne2}. We see
\begin{equation} \label{EqPYOperatorPartOne3}
\begin{split}
\text{LHS of \eqref{EqPYOperatorPartOne1}} &=\Big| \int_{\mathbb{R}^3} f(x-t) \Big( \int_{\mathbb{R}^2} e^{2\pi i \tau P(y)}(\phi_{-s+l-C}\otimes \phi_{-2s+l-C})^{\vee} \big( t-\Gamma(y) \big)K_s(y) dy \Big) dt \Big| \\
&=:|\int_{\mathbb{R}^3} f(x-t) H_{s, l}^{\tau}(t) dt|.
\end{split}
\end{equation}
Moreover, changing the variables $y\mapsto 2^s y$ and using the homogeneous property of $K$, we have
\begin{equation*}
H_{s, l}^{\tau}(t)=\int_{\mathbb{R}^2} e^{2\pi i 2^{3s}\tau P(y)}2^{-4s}(\phi_{l-C}\otimes \phi_{l-C})^{\vee} \big( (\frac{t_{\leq 2}}{2^s}, \frac{t_3}{2^{2s}})-\Gamma(y) \big)K_0(y) dy=2^{-4s}H_{0, l}^{2^{3s}\tau}(\frac{t_{\leq 2}}{2^s}, \frac{t_3}{2^{2s}}).
\end{equation*}
Plug this into RHS of \eqref{EqPYOperatorPartOne3}. We see \eqref{EqPYOperatorPartOne1} is implied by the following estimate for $\tau\sim 2^l$:
\begin{equation}  \label{EqPYOperatorPartOne4}
|H_{0, l}^{\tau}(t)|\lesssim 2^{-cl}(1+|t|)^{-20}.
\end{equation}
To prove \eqref{EqPYOperatorPartOne4}, we plug the Fourier inversion formula of $(\phi_{l-C}\otimes \phi_{l-C})^{\vee}$ into $H_{0, l}^{\tau}$:
\begin{equation}  \label{EqPYOperatorPartOne5}
H_{0, l}^{\tau}(t)=\int_{\mathbb{R}^2} e^{2\pi i \tau P(y)} (\phi_{l-C}\otimes \phi_{l-C})^{\vee} \big( t-\Gamma(y) \big)K_0(y)=\int_{\mathbb{R}^3} e^{2\pi i \xi\cdot t} \phi(\frac{\xi_{\leq 2}}{2^{l-C}}) \phi(\frac{\xi_3}{2^{l-C}}) b(\xi) d\xi,
\end{equation}
where
\begin{equation}  \label{EqPYOperatorPartOne6}
b(\xi):=\int_{\mathbb{R}^2} e^{2\pi i \big( \tau P(y)-\xi\cdot \Gamma(y) \big) }K_0(y) dy.
\end{equation}
In the support of $\phi(\frac{\xi_{\leq 2}}{2^{l-C}}) \phi(\frac{\xi_3}{2^{l-C}})$ in RHS of \eqref{EqPYOperatorPartOne5}, we have
\begin{equation}  \label{EqPYOperatorPartOne7}
\big| \nabla_{y} \big( \tau P(y)-\xi\cdot \Gamma(y) \big) \big| \geq c_2|\tau|-O(|\xi|)\gtrsim 2^l,\ \forall |y|\sim 1.
\end{equation}
Thus integration by parts in \eqref{EqPYOperatorPartOne6} gives
\begin{equation*}
|\partial^{\alpha} b(\xi)|\lesssim_{N'} 2^{-lN'},\ \forall N'\in \mathbb{N}_+.
\end{equation*}
Then integration by parts in \eqref{EqPYOperatorPartOne5} gives \eqref{EqPYOperatorPartOne4}, which finishes the proof of \eqref{EqPYOperatorPartOne}.

For $T_l^{\Upsilon_{>, \neq}}f$, we prove the similar estimate
\begin{equation} \label{EqPYOperatorPartTwo}
|T_l^{\Upsilon_{>, \neq}}f(x)|\lesssim 2^{-cl} M^{\text{str}} f(x),\ \forall x\in \mathbb{R}^3.
\end{equation}
Fix $x$, abbreviate $\tau_x$ as $\tau$, and fix $s$ such that $|\tau|\sim 2^{-3s+l}$ as above. Note that
\begin{equation*}
\sum_{(k, k')\in \Upsilon_{>, \neq}} \psi_{-s+k} \otimes \psi_{-2s+k'}=\sum_{k\geq l+C} \psi_{-s+k} \otimes \phi_{-2s+k-C}+\sum_{k'\geq l+C} \phi_{-s+k'-C} \otimes \psi_{-2s+k'},
\end{equation*}
so it suffices to prove the following for $k\geq l+C$ and $k'\geq l+C$:
\begin{align}
\big| \int_{\mathbb{R}^2} \big( \psi_{-s+k}(D_{\leq 2}) \otimes \phi_{-2s+k-C} (D_3)f \big) \big( x-\Gamma(y) \big) e^{2\pi i \tau_x P(y)} K_s(y) \big| &\lesssim 2^{-ck} M^{\text{str}} f(x), \label{EqPYOperatorPartTwo1} \\
\big| \int_{\mathbb{R}^2} \big( \phi_{-s+k'-C}(D_{\leq 2}) \otimes \psi_{-2s+k'} (D_3)f \big) \big( x-\Gamma(y) \big) e^{2\pi i \tau_x P(y)} K_s(y) \big| &\lesssim 2^{-ck'} M^{\text{str}} f(x). \label{EqPYOperatorPartTwo2}
\end{align}
We prove \eqref{EqPYOperatorPartTwo1}\eqref{EqPYOperatorPartTwo2} using the same argument as in the $T_l^{\Upsilon_{<}}f$ part above. The only difference is that, instead of \eqref{EqPYOperatorPartOne7}, we use
\begin{equation*}
\big| \nabla_{y} \big( \tau P(y)-\xi\cdot \Gamma(y) \big) \big| \geq c_2|\xi_{\leq 2}|-O(|\tau|+|\xi_3|)\gtrsim 2^k,\ \forall |y|\sim 1.
\end{equation*}
for \eqref{EqPYOperatorPartTwo1} and
\begin{equation*}
\big| \nabla_{y} \big( \tau P(y)-\xi\cdot \Gamma(y) \big) \big| \geq c_2|\xi_3|-O(|\tau|+|\xi_{\leq 2}|)\gtrsim 2^{k'},\ \forall |y|\sim 1.
\end{equation*}
for \eqref{EqPYOperatorPartTwo2}. This finishes the proof of \eqref{EqPYOperatorPartTwo}.

For $T_l^{\Upsilon_{>, =}}f$, we write
\begin{equation*} 
\int_{\mathbb{R}^3} \big| \sum_{|s|\leq N}\sum_{(k, k')\in \Upsilon_{>, =}} 1_{|\tau_x|\sim 2^{-3s+l}} \int_{\mathbb{R}^2} \big( \psi_{-s+k}(D_{\leq 2}) \otimes \psi_{-2s+k'} (D_3)f \big) \big( x-\Gamma(y) \big) e^{2\pi i \tau_x P(y)} K_s(y) \big|^2
\end{equation*}
as
\begin{equation}\label{EqPYOperatorPartThree}  
\int_{\mathbb{R}^3} \sum_{|s|\leq N} \big| 1_{|\tau_x|\sim 2^{-3s+l}} \int_{\mathbb{R}^2} \big( \sum_{(k, k')\in \Upsilon_{>, =}}  \psi_{-s+k}(D_{\leq 2}) \otimes \psi_{-2s+k'} (D_3)f \big) \big( x-\Gamma(y) \big) e^{2\pi i \tau_x P(y)} K_s(y) \big|^2.
\end{equation}
For each $s$, applying a rescaled version of Lemma \ref{LemmaContinuousMajorArc} gives
\begin{equation*}
\begin{split}
\int_{\mathbb{R}^3} \big| 1_{|\tau_x|\sim 2^{-3s+l}} \int_{\mathbb{R}^2} & \big( \sum_{(k, k')\in \Upsilon_{>, =}}  \psi_{-s+k}(D_{\leq 2}) \otimes \psi_{-2s+k'} (D_3)f \big) \big( x-\Gamma(y) \big) e^{2\pi i \tau_x P(y)} K_s(y) \big|^2 \\
&\lesssim 2^{-2cl} \int_{\mathbb{R}^3} \big| \big( \sum_{(k, k')\in \Upsilon_{>, =}}  \psi_{-s+k}(D_{\leq 2}) \otimes \psi_{-2s+k'} (D_3)f \big)   \big|^2 \\
&\lesssim 2^{-2cl} \int_{\mathbb{R}^3} \sum_{(k, k')\in \Upsilon_{>, =}} \big| \big( \psi_{-s+k}(D_{\leq 2}) \otimes \psi_{-2s+k'} (D_3)f \big) \big|^2,
\end{split}
\end{equation*}
where we used Plancherel and frequency orthogonality in the last equation. Thus
\begin{equation}  \label{EqPYOperatorPartThree1}
\text{RHS of \eqref{EqPYOperatorPartThree}}\lesssim 2^{-2cl} \int_{\mathbb{R}^3} \sum_{|s|\leq N, (k, k')\in \Upsilon_{>, =}} \big| \big( \psi_{-s+k}(D_{\leq 2}) \otimes \psi_{-2s+k'} (D_3)f \big) \big|^2.
\end{equation}
When $(s, k_1, k_2)\in [-N, N]\times \Upsilon_{>, =}$ varies, the Fourier supports of $\psi_{-s+k}\otimes \psi_{-2s+k'}$ are boundedly overlapping because the following system has $O(1)$ solutions for any given $a, a'$:
\begin{equation*}
\begin{cases}
|(-s+k)-a|<100 \\
|(-2s+k')-a'|<100 \\
|k-k'|<C.
\end{cases}
\end{equation*}
Hence using Plancherel and frequency orthogonality yields
\begin{equation*}
\text{RHS of \eqref{EqPYOperatorPartThree1}}\lesssim 2^{-2cl} \int_{\mathbb{R}^3} |f|^2,
\end{equation*}
which means $\|T_l^{\Upsilon_{>, =}}f\|_2\lesssim 2^{-cl}\|f\|_2$ as desired.

For $T_l^{\Upsilon_{=}}f$, use the same argument as in the $T_l^{\Upsilon_{>, =}}f$ part above. We get that
\begin{equation*} 
\int_{\mathbb{R}^3} \big| \sum_{|s|\leq N}\sum_{(k, k')\in \Upsilon_{=}} 1_{|\tau_x|\sim 2^{-3s+l}} \int_{\mathbb{R}^2} \big( \psi_{-s+k}(D_{\leq 2}) \otimes \psi_{-2s+k'} (D_3)f \big) \big( x-\Gamma(y) \big) e^{2\pi i \tau_x P(y)} K_s(y) \big|^2
\end{equation*}
is controlled by
\begin{equation}\label{EqPYOperatorPartFour}  
2^{-2cl} \sum_{|s|\leq N} \int_{\mathbb{R}^3} \big| \big( \sum_{(k, k')\in \Upsilon_=}  \psi_{-s+k}(D_{\leq 2}) \otimes \psi_{-2s+k'} (D_3)f \big) \big|^2.
\end{equation}
The Fourier support of $\sum_{(k, k')\in \Upsilon_=}  \psi_{-s+k}\otimes \psi_{-2s+k'}$ is contained in
\begin{equation}  \label{EqPYOperatorPartFour1}
\{\xi: \frac{1}{100}\cdot 2^{-s+l-C} \leq |\xi_{\leq 2}|\leq 100\cdot 2^{-s+l+C}\}\cup \{\xi: \frac{1}{100}\cdot 2^{-2s+l-C} \leq |\xi_3|\leq 100\cdot 2^{-2s+l+C}\}.
\end{equation}
When $s$ varies, each strip in \eqref{EqPYOperatorPartFour1} varies in a boundedly overlapping manner. Hence the Fourier supports of $\sum_{(k, k')\in \Upsilon_=}  \psi_{-s+k}\otimes \psi_{-2s+k'}$ are boundedly overlapping. From Plancherel and frequency orthogonality, we deduce
\begin{equation*}
\text{RHS of \eqref{EqPYOperatorPartFour}}\lesssim 2^{-2cl} \int_{\mathbb{R}^3} |f|^2,
\end{equation*}
which means $\|T_l^{\Upsilon_=}f\|_2\lesssim 2^{-cl}\|f\|_2$ as desired. This completes the proof.
\end{proof}

\section{A Bi-parameter Pierce--Yung Operator}  \label{SectionPYBi-Parameter}
We prove Theorem \ref{TheoremPYBi-Parameter} by extending the argument in Section \ref{SectionPY}. By the standard linearizing argument, it suffices to show the following operator is $L^p$ bounded for any fixed measurable functions $\tau_{1, (\cdot)}, \tau_{2, (\cdot)}: \mathbb{R}^3\rightarrow \mathbb{R}$:
\begin{equation*}
\int_{\mathbb{R}^2} f \big( x-\Gamma(y) \big) e^{2\pi i (\tau_{1, x}y_1^3+\tau_{2, x}y_2^3)}  K(y) dy.
\end{equation*}
We can also assume $|\tau_{1, x}|\geq |\tau_{2, x}|$ holds for any $x\in \mathbb{R}^3$. Apply the same decomposition in Section \ref{SectionPY}. It suffices to prove $L^p$ boundedness of the following operator uniformly in $N$:
\begin{equation*}
\sum_{|s|\leq N} \int_{\mathbb{R}^2} f \big( x-\Gamma(y) \big) e^{2\pi i (\tau_{1, x}y_1^3+\tau_{2, x}y_2^3)} K_s(y).
\end{equation*}
Then we decompose our operator into the bi-parameter oscillatory, the one-parameter oscillatory, and the non-oscillatory parts:
\begin{equation}  \label{EqPYBi-ParameterOscillatoryAndNon-Oscillatory}
\begin{split}
\sum_{|s|\leq N} & \int_{\mathbb{R}^2} f \big( x-\Gamma(y) \big) e^{2\pi i (\tau_{1, x}y_1^3+\tau_{2, x}y_2^3)} K_s(y)  \\
&=\sum_{|s|\leq N} 1_{|\tau_{2, x}|>2^{-3s}}\int_{\mathbb{R}^2} \cdots+\sum_{|s|\leq N} 1_{|\tau_{1, x}|>2^{-3s}\geq |\tau_{2, x}|}\int_{\mathbb{R}^2} \cdots+\sum_{|s|\leq N} 1_{|\tau_{1, x}|\leq 2^{-3s}}\int_{\mathbb{R}^2} \cdots.
\end{split}
\end{equation}
As in Section \ref{SectionPY}, the third part (i.e. non-oscillatory part) in RHS of \eqref{EqPYBi-ParameterOscillatoryAndNon-Oscillatory} can be handled by the same argument for the second part (i.e. non-oscillatory part) in RHS of \eqref{EqPYOscillatoryAndNon-Oscillatory}.
For the second part (i.e. one-parameter oscillatory part) in RHS of \eqref{EqPYBi-ParameterOscillatoryAndNon-Oscillatory}, we compare
\begin{equation*}
\big| \int_{\mathbb{R}^2} f \big( x-\Gamma(y) \big) e^{2\pi i (\tau_{1, x}y_1^3+\tau_{2, x}y_2^3)} K_s(y)-\int_{\mathbb{R}^2} f \big( x-\Gamma(y) \big) e^{2\pi i \tau_{1, x}y_1^3} K_s(y) \big| \lesssim (|\tau_{2, x}|2^{3s}) M^{\Gamma}f(x).
\end{equation*}
Hence it suffices to show the following operator
\begin{equation}  \label{EqPYBi-ParameterOne-ParameterOscillatory}
\sum_{|s|\leq N} 1_{|\tau_{1, x}|>2^{-3s}\geq |\tau_{2, x}|}\int_{\mathbb{R}^2} f \big( x-\Gamma(y) \big) e^{2\pi i \tau_{1, x}y_1^3} K_s(y)
\end{equation}
is $L^p$ bounded. We have
\begin{equation*}
\begin{split}
|\eqref{EqPYBi-ParameterOne-ParameterOscillatory}|&=\big| \sum_{l\geq 1} \sum_{s: |s|\leq N, 2^{-3s}\geq |\tau_{2, x}|} 1_{|\tau_{1, x}|\sim 2^{-3s+l}} \int_{\mathbb{R}^2} f \big( x-\Gamma(y) \big) e^{2\pi i \tau_{1, x}y_1^3} K_s(y) \big|  \\
&\leq \sum_{l\geq 1} \sup_{s\in \mathbb{Z}}\sup_{|\tau|\sim 2^{-3s+l}} |\int_{\mathbb{R}^2} f \big( x-\Gamma(y) \big) e^{2\pi i \tau y_1^3} K_s(y)|.
\end{split}
\end{equation*}
For each $l\geq 1$, we apply Lemma \ref{LemmaMaximalMajorArc} with $P_3(y):=y_1^3$. (Lemma \ref{LemmaMaximalMajorArc} extends to any $1<p<\infty$ by interpolation.) Then summing over $l\geq 1$ yields
\begin{equation*}
\|\eqref{EqPYBi-ParameterOne-ParameterOscillatory}\|_p\lesssim \|f\|_p.
\end{equation*}

It remains to tackle the first part (i.e. bi-parameter oscillatory part) in RHS of \eqref{EqPYBi-ParameterOscillatoryAndNon-Oscillatory}. Further decompose it as follows:
\begin{equation*}
\begin{split}
& \text{first part of \eqref{EqPYBi-ParameterOscillatoryAndNon-Oscillatory}} =\sum_{|s|\leq N} \sum_{l_1\geq l_2\geq 1} 1_{|\tau_{1, x}|\sim 2^{-3s+l_1}, |\tau_{2, x}|\sim 2^{-3s+l_2}}\int_{\mathbb{R}^2} f \big( x-\Gamma(y) \big) e^{2\pi i (\tau_{1, x}y_1^3+\tau_{2, x}y_2^3)} K_s(y) \\
&\ \ \ \ =\sum_{l_1\geq l_2\geq 1} \Big( \sum_{|s|\leq N} 1_{|\tau_{1, x}|\sim 2^{-3s+l_1}, |\tau_{2, x}|\sim 2^{-3s+l_2}}\int_{\mathbb{R}^2} f \big( x-\Gamma(y) \big) e^{2\pi i (\tau_{1, x}y_1^3+\tau_{2, x}y_2^3)} K_s(y) \Big) \\
&\ \ \ \ =:\sum_{l_1\geq l_2\geq 1} \mathfrak{T}_{l_1, l_2} f(x).
\end{split}
\end{equation*}
Theorem \ref{TheoremPYBi-Parameter} is then implied by the following estimate:
\begin{proposition}  \label{PropositionPYBi-ParameterMaximalMajorArc}
For $1<p<\infty$ and $l_1\geq l_2\geq 1$, we have
\begin{equation*}
\|\mathfrak{T}_{l_1, l_2} f\|_p\lesssim 2^{-c(p)l_1}\|f\|_p.
\end{equation*}
\end{proposition}
\begin{proof}
The argument is very similar to the proof of Proposition \ref{PropositionPYMaximalMajorArcWeaker}. First, we only need to show the $p=2, l_1\gtrsim 1$ case, and we can assume $\hat{f}$ is supported in a compact set in $\mathbb{R}^3\backslash \{\xi_{\leq 2}=0\text{ or } \xi_3=0\}$. Use the same decomposition to write $\mathfrak{T}_{l_1, l_2}f$ as
\begin{equation}  \label{EqPYBi-ParameterMainDecomposition}
\sum_{|s|\leq N}\sum_{(k, k')\in \mathbb{Z}^2} 1_{|\tau_{j, x}|\sim 2^{-3s+l_j}}\int_{\mathbb{R}^2} \big( \psi_{-s+k}(D_{\leq 2}) \otimes \psi_{-2s+k'} (D_3)f \big) \big( x-\Gamma(y) \big) e^{2\pi i (\tau_{1, x}y_1^3+\tau_{2, x}y_2^3)} K_s(y).
\end{equation}
\end{proof}
For any $\Upsilon\subset \mathbb{Z}^2$, we define $\mathfrak{T}_{l_1, l_2}^{\Upsilon} f$ to be the operator in \eqref{EqPYBi-ParameterMainDecomposition} with $\sum_{(k, k')\in \mathbb{Z}^2}$ replaced by $\sum_{(k, k')\in \Upsilon}$. Fix a large absolute constant $C$. We partition $\mathbb{Z}^2$ into the following four subsets
\begin{align*}
\Upsilon_{<}&:=\big{\{} (k, k')\in \mathbb{Z}^2: \max\{k, k'\}\leq l_2-C \big{\}}, \\
\Upsilon_{>, \neq}&:=\big{\{} (k, k')\in \mathbb{Z}^2: \max\{k, k'\}\geq l_1+C, |k-k'|\geq C \big{\}}, \\
\Upsilon_{>, =}&:=\big{\{} (k, k')\in \mathbb{Z}^2: \max\{k, k'\}\geq l_1+C, |k-k'|<C \big{\}}, \\
\Upsilon_{=}&:=\big{\{} (k, k')\in \mathbb{Z}^2: l_2-C<\max\{k, k'\}<l_1+C \big{\}}.
\end{align*}
Decompose $\mathfrak{T}_{l_1, l_2} f=\mathfrak{T}_{l_1, l_2}^{\Upsilon_{<}}f+\mathfrak{T}_{l_1, l_2}^{\Upsilon_{>, \neq}}f+\mathfrak{T}_{l_1, l_2}^{\Upsilon_{>, =}}f+\mathfrak{T}_{l_1, l_2}^{\Upsilon_{=}}f$ accordingly. We shall prove $\|\mathfrak{T}_{l_1, l_2}^{\Upsilon}f\|_2\lesssim 2^{-\frac{c}{2}l_1}\|f\|_2$ for $\Upsilon=\Upsilon_{<}, \Upsilon_{>, \neq}, \Upsilon_{>, =}, \Upsilon_{=}$ separately.

For $\mathfrak{T}_{l_1, l_2}^{\Upsilon_{<}}f$, it suffices to show
\begin{equation}  \label{EqPYBi-ParameterOperatorPartOne}
|\mathfrak{T}_{l_1, l_2}^{\Upsilon_{<}}f(x)|\lesssim 2^{-cl_1} M^{\text{str}} f(x),\ \forall x\in \mathbb{R}^3.
\end{equation}
By the same argument as in the proof of Proposition \ref{PropositionPYMaximalMajorArcWeaker}, we only need to prove that the kernel
\begin{equation}  \label{EqPYBi-ParameterOperatorPartOne1}
H_{0, l}^{\tau_1, \tau_2}(t):=\int_{\mathbb{R}^2} e^{2\pi i( \tau_1 y_1^3+\tau_2 y_2^3)} (\phi_{l_2-C}\otimes \phi_{l_2-C})^{\vee} \big( t-\Gamma(y) \big)K_0(y)
\end{equation}
satisfies
\begin{equation}  \label{EqPYBi-ParameterOperatorPartOne2}
|H_{0, l}^{\tau_1, \tau_2}(t)|\lesssim 2^{-cl_1}(1+|t|)^{-20}.
\end{equation}
for $\tau_j\sim 2^{l_j}$. To prove \eqref{EqPYBi-ParameterOperatorPartOne2}, we plug the Fourier inversion formula of $(\phi_{l_2-C}\otimes \phi_{l_2-C})^{\vee}$ into $H_{0, l}^{\tau_1, \tau_2}(t)$:
\begin{equation}  \label{EqPYBi-ParameterOperatorPartOne3}
H_{0, l}^{\tau_1, \tau_2}(t)=\int_{\mathbb{R}^3} e^{2\pi i \xi\cdot t} \phi(\frac{\xi_{\leq 2}}{2^{l_2-C}}) \phi(\frac{\xi_3}{2^{l_2-C}}) b(\xi) d\xi,
\end{equation}
where
\begin{equation}  \label{EqPYBi-ParameterOperatorPartOne4}
b(\xi):=\int_{\mathbb{R}^2} e^{2\pi i \big( \tau_1 y_1^3+\tau_2 y_2^3-\xi\cdot \Gamma(y) \big) }K_0(y) dy.
\end{equation}
In the support of $\phi(\frac{\xi_{\leq 2}}{2^{l_2-C}}) \phi(\frac{\xi_3}{2^{l_2-C}})$ on RHS of \eqref{EqPYBi-ParameterOperatorPartOne3}, we have
\begin{align*}
\big| \partial_{j}^{\alpha} \big( ( \tau_1 y_1^3+\tau_2 y_2^3)-\xi\cdot \Gamma(y) \big) \big| & \lesssim |\tau_j|,\ \forall |y|\sim 1, \alpha\in \mathbb{N}_+,  \\
\big| \partial_{j} \big( ( \tau_1 y_1^3+\tau_2 y_2^3)-\xi\cdot \Gamma(y) \big) \big| & \geq c_2|\tau_j|-O(|\xi|) \gtrsim |\tau_j|,\ \forall |y_j|>c_3.
\end{align*}
Decompose $K_0$ into $K_0\chi_1+K_0\chi_2$ such that $K_0\chi_j$ is supported in $\{|y_j|>c_3\}$. Apply integration by parts to the $K_0\chi_j$ piece for $j=1, 2$ separately, and then add together. We see $|\partial^{\alpha} b(\xi)|\lesssim_{N} 2^{-l_2N}, \forall N\in \mathbb{N}_+$. Moreover, we have
\begin{equation*}
\big| \partial_1^3\big(( \tau_1 y_1^3+\tau_2 y_2^3)-\xi\cdot \Gamma(y) \big) \big|=6|\tau_1|\sim 2^{l_1},\ \forall |y|\sim 1.
\end{equation*}
Applying the van der corput lemma \cite[Proposition 2.1]{SWPolynomialCarleson} to \eqref{EqPYBi-ParameterOperatorPartOne4} yields $|\partial^{\alpha} b(\xi)|\lesssim 2^{-2cl_1}$. Taking the geometric average gives
\begin{equation*}
|\partial^{\alpha} b(\xi)|\lesssim_{N} 2^{-cl_1} 2^{-l_2\frac{N}{2}},\ \forall N\in \mathbb{N}_+.
\end{equation*}
Then integration by parts in \eqref{EqPYBi-ParameterOperatorPartOne3} yields \eqref{EqPYBi-ParameterOperatorPartOne2}, which finishes the proof of \eqref{EqPYBi-ParameterOperatorPartOne}.

For $\mathfrak{T}_{l_1, l_2}^{\Upsilon_{>, \neq}}f$, we have the similar estimate
\begin{equation*} 
|\mathfrak{T}_{l_1, l_2}^{\Upsilon_{>, \neq}}f(x)|\lesssim 2^{-cl_1} M^{\text{str}} f(x),\ \forall x\in \mathbb{R}^3.
\end{equation*}
For this part, the argument in the proof of Proposition \ref{PropositionPYMaximalMajorArcWeaker} does not distinguish the one-parameter and the bi-parameter cases.

Similarly, for $\mathfrak{T}_{l_1, l_2}^{\Upsilon_{>, =}}f$, we argue as in the proof of Proposition \ref{PropositionPYMaximalMajorArcWeaker} but apply Lemma \ref{LemmaContinuousMajorArcBi-Parameter} instead. Then obtain
\begin{equation*}
\|\mathfrak{T}_{l_1, l_2}^{\Upsilon_{>, =}}f\|_2\lesssim 2^{-cl_1}\|f\|_2
\end{equation*}
as desired.

For $\mathfrak{T}_{l_1, l_2}^{\Upsilon_=}f$, we argue as in the proof of Proposition \ref{PropositionPYMaximalMajorArcWeaker} as well. The difference is that \eqref{EqPYOperatorPartFour1} should be replaced by
\begin{equation*}
\{\xi: \frac{1}{100}\cdot 2^{-s+l_2-C} \leq |\xi_{\leq 2}|\leq 100\cdot 2^{-s+l_1+C}\}\cup \{\xi: \frac{1}{100}\cdot 2^{-2s+l_2-C} \leq |\xi_3|\leq 100\cdot 2^{-2s+l_1+C}\}.
\end{equation*}
When $s$ varies, each strip above varies with $O(l_1)$ overlapping. By Plancherel, we deduce
\begin{equation*}
\int_{\mathbb{R}^3} |\mathfrak{T}_{l_1, l_2}^{\Upsilon_=}f|^2 \lesssim 2^{-2cl_1}\cdot O(l_1)\cdot \int_{\mathbb{R}^3} |f|^2\lesssim 2^{-cl_1} \int_{\mathbb{R}^3} |f|^2.
\end{equation*}
This completes the proof of Proposition \ref{PropositionPYBi-ParameterMaximalMajorArc}.

\section{Concluding Remarks} \label{SectionFinalRemarks}
\begin{enumerate}
\item In \cite{PYPolynomialCarlesonRadon, AMPYPolynomialCarlesonRadon}, the authors managed to handle much more general operators than \eqref{EqPY}. It seems not straightforward to extend our approach in Section \ref{SectionPY} to their general setting.

\item The following planar version of \eqref{EqPY} cannot be handled by the Pierce—Yung argument or ours:
\begin{equation*}
\sup_{\tau\in \mathbb{R}} |\int_{\mathbb{R}} f(x_1-y, x_2-y^2) e^{2\pi i \tau y^3} K(y) dy|.
\end{equation*}
In fact, the deduction in Section \ref{SectionPY} works equally well. But the $TT^*$ method in Section \ref{SectionMajorArcLemma} does not work in the planar case. This can already be seen in the finite field setting: the analogue of \eqref{EqFiniteFieldsRangeOfy} uniquely determines $y, y'\in \mathbb{F}_q$, and RHS of \eqref{EqFiniteFieldsSummand3} becomes a complex number with norm $1$. In \cite{BGHCarlesonRadon} and \cite{HLCarlesonRadon}, the authors established $L^p$ boundedness of this planar Pierce—Yung operator via different methods.

\item Our two-step paraproduct decomposition in Section \ref{SectionPY} and Section \ref{SectionPYBi-Parameter} is inspired by the one-step paraproduct decomposition in \cite{CDRSquareCornerSingular}. After finding the short approach in Section \ref{SectionPY}, we learned that this type of decomposition was also used in \cite{BGHCarlesonRadon}. However, the two arguments proceed somewhat differently: ours exploits more frequency orthogonality and only needs the decay $2^{-cl}$. See the estimates after \cite[Theorem 3.2]{BGHCarlesonRadon} for some comparison.

\item Using an inverse version of Lemma \ref{LemmaContinuousMajorArc}, it’s not hard to adapt the argument in Section \ref{SectionFindingPatterns} to solve the analogous pattern finding problem in the continuous setting. In fact, with an inverse theorem similar to Theorem \ref{TheoremFindingPatternsInverse}, we can first argue as in \cite{KMPWWMultilinearPolynomialErgodic} to establish a smoothing inequality. Then apply Bourgain’s energy pigeonholing argument as in \cite{CDRSquareCornerSingular} to find patterns (with a gap estimate) from that smoothing inequality. We omitted the details. There are other possible approaches to establish smoothing inequalities in the continuous setting such as \cite{CGLContinuousPolynomialRoth, CDRSquareCornerSingular, HLSquareCornerSmoothing}. Some of them used special properties of the continuous setting, and such approaches might not be suitable for pattern finding problems in finite fields like Theorem \ref{TheoremFindingPatterns}. (For example, \cite{HLSquareCornerSmoothing} used the asymptotic expansion of oscillatory integrals.)
\end{enumerate}

\bibliographystyle{plain}  
\bibliography{ref}

\Address

\end{document}